\documentclass[a4paper,11pt]{article}

\usepackage[T1]{fontenc}
\usepackage{amsmath, graphicx, color, amsthm, centernot, amssymb}
\usepackage[font=small,labelfont=bf,width=12.5cm]{caption}
\usepackage{algorithmicx,algorithm,algpseudocode,url}

\usepackage[labelformat=simple]{subcaption}

\usepackage{lineno}

\usepackage{array}
\newcommand{\PreserveBackslash}[1]{\let\temp=\\#1\let\\=\temp}
\newcolumntype{C}[1]{>{\PreserveBackslash\centering}p{#1}}
\newcolumntype{R}[1]{>{\PreserveBackslash\raggedleft}p{#1}}
\newcolumntype{L}[1]{>{\PreserveBackslash\raggedright}p{#1}}

\def\colwidth{0.5cm}
\def\scaleConst{0.60}

\newtheorem{theorem}{Theorem}
\newtheorem{lemma}{Lemma}
\newtheorem{corollary}{Corollary}
\newtheorem{proposition}{Proposition}
\newtheorem{observation}{Observation}
\newtheorem{conjecture}{Conjecture}

\newtheorem{question}{Question}

\newcommand{\set}[1]{\ensuremath{\left\{#1 \right\}}}

\newcommand{\chis}[1]{%
	\ensuremath{%
		\protect{%
			\chi_{\mathrm{st}}'(#1)%
		}
	}
}

\newcommand\cartp{%
	\mathbin{%
		\ensuremath{\square}	
	}
}

\title{Star Edge-Coloring of Square Grids}
\author
{	
	P\v{r}emysl Holub\thanks{Faculty of Applied Sciences, University of West Bohemia, Pilsen, Czech Republic.\newline
		E-Mails: \texttt{holubpre@kma.zcu.cz, mmockov@ntis.zcu.cz}}, \
		Borut Lu\v{z}ar\thanks{Faculty of Information Studies, Novo mesto, Slovenia. 
		E-Mail: \texttt{borut.luzar@gmail.com}}, \
		Erika Mihalikov\'{a}\thanks{Faculty of Science, Pavol Jozef \v Saf\'{a}rik University, 
		Ko\v{s}ice, Slovakia.\newline
		E-Mails: \texttt{erika.mihalikova@student.upjs.sk, roman.sotak@upjs.sk}}, \\
	Martina Mockov\v{c}iakov\'{a}\footnotemark[1], \
	Roman Sot\'{a}k\footnotemark[2]	
}

\begin{document}
\maketitle

{
	\abstract
	{
		\textit{A star edge-coloring} of a graph $G$ is a proper edge-coloring without bichromatic paths or cycles of length four. 	
		The smallest integer $k$ such that $G$ admits a star edge-coloring with $k$ colors is the \textit{star chromatic index} of $G$.
		In the seminal paper on the topic, Dvo\v{r}\'{a}k, Mohar, and \v{S}\'{a}mal asked if the star chromatic index of complete graphs
		is linear in the number of vertices and gave an almost linear upper bound. 
		Their question remains open, and consequently, to better understand the behavior of the star chromatic index,
		this parameter has been studied for a number of other classes of graphs.
		In this paper, we consider star edge-colorings of square grids;
		namely, the Cartesian products of paths and cycles and the Cartesian products of two cycles.
		We improve previously established bounds and, as a main contribution,
		we prove that the star chromatic index of graphs in both classes is either $6$ or $7$ except for prisms.
		Additionally, we give a number of exact values for many considered graphs.
		
		\bigskip
		{\bf Keywords:} star edge-coloring, star chromatic index, square grid, Cartesian product.
	}
}

\section{Introduction}

A proper edge-coloring of a graph $G$ is called a \textit{star edge-coloring} 
if there is neither bichromatic path nor bichromatic cycle of length four. 
The minimum number of colors for which $G$ admits a star edge-coloring is called the \textit{star chromatic index} and we denote it by $\chis{G}$.

The star edge-coloring was defined in 2008 by Liu and Deng~\cite{LiuDen08}, 
and was motivated by the vertex version introduced by Gr\"{u}nbaum~\cite{Gru73}.
Despite a number of papers have already been published about this coloring,
we have a very limited knowledge about it.
In particular, the exact value of the star chromatic index of complete graphs is still not known,
although some relatively strong lower and upper bounds have been determined by Dvo\v{r}\'{a}k et al.
in their seminal paper~\cite{DvoMohSam13}.
\begin{theorem}[Dvo\v{r}\'{a}k, Mohar, \v{S}\'{a}mal, 2013]
	\label{thm:complete}
	The star chromatic index of the complete graph $K_n$ satisfies
	$$
		2n (1 + o(1)) \le \chis{K_n} \le n \frac{2^{2\sqrt{2}(1+o(1))\sqrt{\log{n}}}}{(\log{n})^{1/4}}\,.
	$$
	In particular, for every $\epsilon >0$ there exists a constant $C$ such that $\chis{K_n} \le C \, n^{1+\epsilon}$ for every $n\ge 1$.
\end{theorem}
\noindent
They proved the upper bound using a nontrivial result about sets without arithmetic progressions, and up till now, it is still the best known.
For the lower bound, they used an elegant double counting approach.
The authors of~\cite{BezLuzMocSotSkr16} observed a small improvement in their proof 
and obtained the bound $\chis{K_n} \ge 3n(n-1)/(n+4)$ (see~\cite{Moc13} for a proof),
which gives the exact values for the chromatic index of $K_n$, for $n \in \set{1,2,3,4,8}$.
However, despite all efforts, the asymptotic behavior of the star chromatic index of complete graphs is not known,
and in~\cite{DvoMohSam13} the following question has been asked.
\begin{question}[Dvo\v{r}\'{a}k, Mohar, \v{S}\'{a}mal, 2013]
	What is the true order of magnitude of $\chis{K_n}$? Is $\chis{K_n} = O(n)$?
\end{question}

Another class of graphs with highly regular structure are complete bipartite graphs.
They are important for better understanding of the coloring already on their own,
and also, as Dvo\v{r}\'{a}k et al.~\cite{DvoMohSam13} observed, 
the bounds for their star chromatic index provide bounds for the index of complete graphs.
$$
	\chis{K_{n,n}} - n \, \le \, \chis{K_n} \, \le \, \sum_{i=1}^{\lceil \log_2 n \rceil} 2^{i-2} \, \chis{K_{\lceil n/2^i \rceil, \lceil n/2^i \rceil}}\,.
$$
Recently, Casselgren et al.~\cite{CasGraRas19} considered complete bipartite graphs 
and proved the tight upper bound for $K_{3,r}, r\ge 5$, derived a lower and upper bound for $K_{4,s}, s\ge 4$,
and, using computer, they also determined the star chromatic index for some complete bipartite graphs of small order.

Star edge-coloring has been studied also for other classes of graphs,
e.g., graphs with maximum degree $3$~\cite{DvoMohSam13,LeiShiSon18,LeiShiSonWan18} and $4$~\cite{WanWanWan19},
subcubic Halin graphs~\cite{CasGraRas19}, 
outerplanar graphs~\cite{BezLuzMocSotSkr16,WanWanWan18}, 
and planar graphs with various constraints~\cite{WanWanWan18}.
Moreover, the list version of the star edge-coloring has also been investigated (see, e.g.,~\cite{KerRas18,LuzMocSot19}).
Finally, there is also a complexity result on the topic; 
namely, it is NP-complete to decide whether $3$ colors suffice for a star edge-coloring of a subcubic multigraph~\cite{LeiShiSon18}.

Since most of the obtained upper bounds for the star chromatic index are not tight and many questions remain open,
we focus our attention to graphs with a relatively simple structure, i.e. to the Cartesian products of graphs. 

The star edge-coloring of the Cartesian products of graphs has already been considered by Omoomi and Dastjerdi~\cite{OmoDas18}.
They established an upper bound for the star chromatic index of the Cartesian product of two arbitrary graphs, 
proved its exact values for the Cartesian product of two paths (Theorem~\ref{thm:PnPn}),
and they started investigation on the Cartesian products of a path and a cycle, and the Cartesian product of two cycles (i.e., square grids).
They further proved upper bounds for $d$-dimensional grids and $d$-dimensional hypercubes.

Motivated by the results presented in~\cite{OmoDas18}, 
in this paper, we consider star edge-coloring of square grids; 
in particular, the Cartesian products of two cycles and the Cartesian products of paths and cycles.
Apart from the usual combinatorial methods,
due to the complexity of the problems considered in this paper, 
we have used computer to obtain star edge-colorings of small graphs and to establish some of the lower bounds.
Standard (formal) mathematical proofs would require enourmous amount of case analysis,
while their contribution to the theory would be minimal. 
We establish exact bounds for the star chromatic index of many graphs from the two considered classes,
and show that the upper bound for the chromatic index of both Cartesian products is $7$.

The paper is structured as follows.
We give our notation and prove some auxiliary results in Section~\ref{sec:prel}.
Section~\ref{sec:alg} contains the algorithm used in our computations
and describes the preprocessing procedures used in them.
In Section~\ref{sec:main}, we present the main results of this paper,
and we list some open problems in Section~\ref{sec:conc}.


\section{Preliminaries}
\label{sec:prel}

In this section, we present some additional terminology used in the paper and give auxiliary results.
We abbreviate a `star edge-coloring with $k$ colors' to a `star $k$-edge-coloring', 
and, if it is clear from the context, sometimes we just write `coloring' instead of `star edge-coloring'.

The \textit{Cartesian product of graphs $G$ and $H$}, denoted by $G \cartp H$, 
is the graph with the vertex set $V(G) \times V(H)$ and edges between the vertices $(u,v)$ and $(u',v')$ if:
\begin{itemize}
	\item $uu'\in E(G)$ and $v=v'$ (a \emph{$G$-edge}), or
	\item $u=u'$ and $vv'\in E(H)$ (an \emph{$H$-edge}).
\end{itemize}
We call the graphs $G$ and $H$ the \textit{factor graphs}.
The \emph{$G$-f\mbox{}iber} with respect to $v\in V(H)$, denoted by $G_v$, is the copy of $G$ in $G \cartp H$ 
induced by the vertices having $v$ as the second component. 
Analogously, the \emph{$H$-f\mbox{}iber} with respect to $u \in V(G)$, 
denoted by $H_u$, is the copy of $H$ in $G \cartp H$ induced by the vertices having $u$ as the first component. 

Since the Cartesian product of two paths is a subgraph of the Cartesian product of a path and a cycle,
and the Cartesian product of a path and a cycle is a subgraph of the Cartesian products two cycles,
we have the following sequence of inequalities.
\begin{observation}
	\label{obs:sub}
	For every pair of positive integers $m$ and $n$, where $m \ge 3$ and $n \ge 3$, we have
	$$
		\chis{P_m \cartp P_n} \le \chis{C_m \cartp P_n} \le \chis{C_m \cartp C_n}.
	$$
\end{observation}

Having a star edge-coloring of the Cartesian product of an $n$-cycle and a graph $H$,
we can extend it to a coloring of the Cartesian product of a cycle of length $k\cdot n$ and $H$.
\begin{lemma}
	\label{lem:tiling}
	For every integers $k$ and $m$, where $k \ge 2$ and $m \ge 3$, and for every graph $H$, we have
	$$
		\chis{C_{k \cdot m} \cartp H} \le \chis{C_m \cartp H}\,.
	$$
\end{lemma}

\begin{proof}
	Let $C_m = u_1 \dots u_m u_1$, 
	$C_{k\cdot m} = v_1 \dots v_{k \cdot m} v_{1}$, 
	and $V(H)=\{w_1,\dots,w_n\}$. 
	Moreover, let $p:\{1,\dots,k \cdot m\} \to \{1,\dots,m\}$ be an assignment given by $p(t)=s$ if and only if $(t-s)$ is divisible by $m$.

	Let $\sigma$ be a star edge-coloring of $C_m \cartp H$.	
	Consider an edge $e=(v_i,w_a)(v_j,w_b)$ of $C_{k \cdot m} \cartp H$. 
	By the definition of the Cartesian product, we have $i=j$ or $a=b$. 
	Note that if $a = b$, then $|i-j|=1$, and thus we may assume $j=i+1$.	
	We define a proper edge-coloring $\tau$ of $C_{k \cdot m} \cartp H$ as follows.	
	If $i=j$, then set $\tau((v_i,w_a)(v_i,w_b)) = \sigma((u_{p(i)},w_a)(u_{p(i)},w_b))$.
	In the case $a=b$, we set	
	$\tau((v_i,w_a)(v_j,w_a))=\sigma((u_{p(i)},w_a)(u_{p(j)},w_a))$.
	
	Now we show that $\tau$ is also a star edge-coloring.
	For an integer $s$, where $1 \le s \le k \cdot m$,
	let $G_{s}$ be the graph induced by the vertices $\{(v_\ell,w)\}$, 
	where $\ell \in \set{s+1,\dots,s+m}$ (the values $s+1, \dots, s+m$ are taken modulo $k \cdot m$) and all $w \in V(H)$,
	i.e., $G_s$ is the graph induced on $m$ consecutive $H$-fibers.
	Observe that the coloring $\tau$ on $G_s$ corresponds to a coloring $\sigma$
	of the subgraph of $C_{m}\square H$ without the edges $(u_{p(s+m)},w)(u_{p(s+1)},w)$, for all $w\in V(H)$. 
	Therefore, every $4$-path and every $4$-cycle in $C_{k \cdot m} \cartp H$, contained in some $G_s$,
	is not bichromatic.
	
	Finally, if a $4$-path or a $4$-cycle is not contained in any $G_s$, then it contains at least $m$ edges
	of type $(v_i,w_a)(v_j,w_a)$ (i.e., only when $m \in \set{3,4}$).
	However, in the case of $m=3$, three consecutive edges on every $C_m$-fiber receive three distinct colors, 
	and hence no $4$-path with three consecutive edges on a $C_m$-fiber is bichromatic.
	If a $4$-path has two consecutive edges on a $C_m$-fiber, an edge in an $H$-fiber, and the fourth edge in another $C_m$-fiber,
	then its coloring corresponds to a coloring of some $4$-path in $C_m \cartp H$, which is not bichromatic.
	In the case of $m=4$, we only have $4$-cycles, whose colorings correspond to a coloring of a $4$-cycle by
	$\sigma$ in some $C_m$-fiber,
	and hence they are not bichromatic.
\end{proof}

We continue by showing how star edge-colorings of two Cartesian products,
each having at least one cycle as a factor, can be combined.
Let $m$ and $n$ be a pair of integers, where $3 \le m < n$,
and let $v_1,\dots,v_n$ be consecutive vertices of the cycle $C_n$.
We say that a star edge-coloring $\sigma$ of $C_n \cartp H$ \textit{includes} a star edge-coloring of $C_m \cartp H$ 
if the coloring $\sigma^*$ of the subgraph of $C_n \cartp H$ 
induced by the vertices of $m$ consecutive $H$-fibers $H_{v_1},\dots,H_{v_m}$, 
together with the additional edges $e_w=(v_1,w)(v_m,w)$, for all $w \in V(H)$,
where we set $\sigma^*(e_w) = \sigma((v_1,w)(v_n,w))$, is a star edge-coloring.

Symmetrically, we can say that a star $k$-edge-coloring of $H \cartp C_n$ includes a star edge-coloring of $H \cartp C_m$.
Note that the star edge-coloring of $C_{k \cdot m} \cartp H$, constructed in the proof of Lemma~\ref{lem:tiling}, 
includes a star edge-coloring of $C_m \cartp H$.

\begin{lemma}
	\label{lem:combine}
	If for a pair of positive integers $m$ and $n$, where $m < n$, 
	a star $k$-edge-coloring of $C_n \cartp H$ includes a star edge-coloring of $C_m \cartp H$, 
	then, for every pair of non-negative integers $p$ and $q$, we have
	$$
		\chi_{st}^{\prime}(C_{p \cdot m + q \cdot n} \cartp H) \le k\,.
	$$	
\end{lemma}

\begin{proof}
	Let $\sigma$ be a star $k$-edge-coloring of $C_n \cartp H$ which includes a star edge-coloring $\sigma^*$ of $C_m \cartp H$. 
	Let $C_n = v_1\dots v_nv_1$, $C_m=v_1\dots v_mv_1$, and $C_{pm+qn}=u_1\dots u_{pm+qn}u_1$. 
	Furthermore, we define an assignment $r:\{1,\dots,pm+qn\}\to\{1,\dots,n\}$ 
	such that, 
	if $t \le pm$, then $r(t) \le m$ and $t-r(t)$ is divisible by $m$, 
	and, 
	if $t > pm$, then $t-pm-r(t)$ is divisible by $n$.

	Now, similarly as in the proof of Lemma~\ref{lem:tiling}, 
	we define an edge-coloring $\tau$ of $C_{pm+qn} \cartp H$. 
	We combine $p$ copies of $\sigma^*$ followed by $q$ copies of $\sigma$. 
	More precisely, for $w_a,w_b \in H$,
	$$
		\tau((u_i,w_a)(u_j,w_b)) =
			\begin{cases}
				\sigma^*((v_{r(i)},w_a)(v_{r(i)},w_b)), &\mbox{ for }i=j\le pm\,;\\
				\sigma((v_{r(i)},w_a)(v_{r(i)},w_b)), &\mbox{ for }i=j>pm\,; \\
				\sigma^*((v_{r(i)},w_a)(v_{r(j)},w_a)), &\mbox{ for }i=j-1\le pm\,; \\
				\sigma((v_{r(i)},w_a)(v_{r(j)},w_a)), &\mbox{ for }i=j-1>pm\,.
			\end{cases}
	$$
	Note that, by the definition of $\sigma^*$, we have $\sigma^*((v_m,w)(v_1,w)) = \sigma((v_n,w)(v_1,w))$. 

	It remains to show that $\tau$ is a star edge-coloring.
	For an integer $s$, where $1 \le s \le pm+qn$, 
	let $G_s$ be the graph induced by the vertices of $m+1$ consecutive $H$-fibers $H_{u_{s+1}},\dots,H_{u_{s+m+1}}$ 
	(the indices $s+i$ are taken modulo $pm+qn$). 
	Since the coloring of each $G_s$ is a part of two consecutive copies of $\sigma^*$ or a part of two consecutive copies of $\sigma$,
	no $4$-path and no $4$-cycle in $G_s$ is bichromatic (using Lemma~\ref{lem:tiling} for $k=2$). 

	Finally, if a $4$-path is not contained in any $G_s$, then it contains $4$ edges of type $(v_i,w_a)(v_j,w_a)$ and $m = 3$.
	Moreover, such a $4$-path traverses the $H_{v_i}$-fibers, for $i \in \set{pm,\dots,pm+4}$ or $i \in \set{pm+qn,1,\dots,4}$.
	In both cases, colors of three consecutive edges of the 4-path correspond to colors of a $C_3$-fiber of $\sigma^*$ and therefore they are distinct.
	This completes the proof.
\end{proof}

We will use Lemma~\ref{lem:combine} to prove results for arbitrary lengths of cycles.
To do that, we will use the following result on Frobenious numbers~\cite{Syl82}.
\begin{theorem}[Sylvester, 1882]
	\label{thm:frob}
	Let positive integers $n$ and $m$ be relatively prime. 
	Then for every integer $k \ge (n-1)(m-1)$ there exist non-negative integers $\alpha$ and $\beta$ such that
	$$
		k = \alpha \cdot n + \beta \cdot m \,.
	$$	
\end{theorem}

We also recall the result of Dvo\v{r}\'{a}k, Mohar, and \v{S}\'{a}mal~\cite{DvoMohSam13} about star edge-coloring of subcubic graphs,
which we will use when considering prisms.
\begin{theorem}[Dvo\v{r}\'{a}k, Mohar, \v{S}\'{a}mal, 2013]
	\label{thm:cubic}
	~
	\begin{itemize}
		\item[(a)] If $G$ is a subcubic graph, then $\chis{G} \le 7$.
		\item[(b)] If $G$ is a simple cubic graph, then $\chis{G} \ge 4$, 
			and the equality holds if and only if $G$ covers the graph of the $3$-dimensional hypercube. 
	\end{itemize}
\end{theorem}
\noindent
Here, a graph $G$ is said to \textit{cover} a graph $H$ if there is a graph homomorphism from $G$ to $H$ that is locally bijective. 
In other words, there is a mapping $f : V(G) \rightarrow V(H)$
such that whenever $uv$ is an edge of $G$, the image $f(u)f(v)$ is an edge of $H$, 
and, for each vertex $v \in V(G)$, $f$ is a bijection between the neighbors of $v$ and the neighbors of $f(v)$.

At this point, we remark the following, somehow hidden, corollary of the above result.
Hexagonal grids are subcubic graphs and they cover the graph of the $3$-dimensional hypercube.
Thus:
\begin{corollary}
	For an infinite hexagonal grid $G$, we have
	$$
		\chis{G} = 4\,.
	$$
\end{corollary}

%

\section{Computer Computations and Algorithm}
\label{sec:alg}

For our computations, we used a simple backtracking algorithm,
which, together with some preprocessing, enabled us to compute exact lower bounds for some important cases on one hand,
and on the other hand, provided star edge-colorings with required properties for some graphs.
\begin{algorithm}[ht]
    \caption{Star edge-coloring algorithm}
    \label{alg:color}
    \begin{algorithmic}[1]
        \Procedure{StarColor}{$G,k,\mathcal{P}$} \Comment{Graph $G$, number of colors $k$, precolored edges $\mathcal{P}$}                        
            \State edgeOrder $\gets$ GetEdgeOrdering($G$, $\mathcal{P}$)
            \State edgeColors $\gets$ InitEdgeColors($\mathcal{P}$,edgeOrder)
            \State triedColors $\gets$ InitTriedColors(edgeOrder) \Comment{Dictionary of empty lists for all edges}
            \For{$i$ in $1..\textrm{edgeOrder.Count}$}\label{for:Edges} \Comment{Try to color edges according to the ordering}
				\State $e$ $\gets$ edgeOrder[$i$]
				\State isColored $\gets$ \texttt{false}
				\For{color $c$ in $\set{1..k}$ $\setminus$ triedColors[$e$]}\label{for:Colors}
					\State edgeColors[$e$] $\gets$ $c$
					\If{Conflict($c$, $G$, edgeColors)} \Comment{Check if a conflict occurs}
						\State edgeColors[$e$] $\gets$ $\emptyset$
					\Else
						\State add $c$ to triedColors[$e$]
						\State isColored $\gets$ \texttt{true}
						\State goto \ref{if:out}
					\EndIf
                \EndFor
                \If{!isColored and $i > 1$}\label{if:out} \Comment{If no color is found, continue if not at first edge}
                	\State edgeColors[edgeOrder[$i-1$]] $\gets$ $\emptyset$
                	\State $i = i-2$ \Comment{Step up, $-2$ handles automatic loop increment}
                \ElsIf{!isColored and $i == 1$}
                	\State \textbf{return} ``No coloring found''
                \ElsIf{isColored and $i == \textrm{edgeOrder.Count}$}
                	\State \textbf{return} edgeColors \Comment{A star edge-coloring is found}
                \EndIf
            \EndFor            
        \EndProcedure
    \end{algorithmic}
\end{algorithm}

The main coloring algorithm takes three input parameters: the graph to be colored, the number of colors, 
and a possible precoloring of some edges, in order to avoid testing some isomorphic partial colorings;
e.g., one may fix the colors on the edges incident to a vertex of maximum degree.
Note also that before calling the function \texttt{StarColor}, we first verify that the precoloring of the edges is a star edge-coloring.

Another important part of our algorithm is determining the order of edges (the function \texttt{GetEdgeOrdering}), in which it tries to color them.
We order the edges (ignoring the precolored edges) by the number of precolored neighbors (incident edges) 
and the number of neighbors appearing earlier in the ordering in a descending order.

The function \texttt{Conflict} checks if assigning a color to the current edge introduces a conflict,
namely, it checks if two adjacent edges receive the same color, and if a bichromatic $4$-path or $4$-cycle appears.
In some cases, we manually controlled the different cases of precolored edges. 
If we established some additional property of a required coloring, e.g., that no $4$-path can be colored with just three colors,
we included that in the procedure.

Finally, we also adopted the algorithm to output all possible colorings of a given graph, and in the case of symmetric graphs, e.g., cycles,
we eliminated isomorphic colorings. The remaining colorings were used to test if they can be extended to graphs on more vertices.

\newpage

\section{Cartesian products of paths and cycles}
\label{sec:main}


\subsection{Cartesian products of paths}
\label{sub:PnPn}

In a recent paper, Omoomi and Dastjerdi~\cite{OmoDas18} established tight bounds for the star chromatic index of two paths (see Table~\ref{tbl:squaregrid}). 
\begin{theorem}[Omoomi and Dastjerdi, $2019$]
	\label{thm:PnPn}
	For the graph $P_m \cartp P_n$, where $m$ and $n$ are integers with $2 \le m \le n$, we have
	$$
		\chis{P_m \cartp P_n} = 
			\begin{cases} 
				3, &\text{if $m = n = 2$;} \\
				4, &\text{if $m = 2$ and $n \ge 3$;} \\
				5, &\text{if $m = 3$ and $3 \le n \le 4$;} \\
				6, &\text{otherwise.}
			\end{cases}
	$$	
\end{theorem}

\begin{table}[htb] 
\begin{center}
	\begin{tabular}{|c||c|c|c|c|}
	\hline
	$m \backslash n$ & 	2 	& 	3 	& 	4	&	$5^+$ \\
	\hline
	\hline
	2				& 	3	&	4 	&	4	&	4 \\
	\hline
	3				& 	4	&	5 	&	5	&	6 \\
	\hline
	4				& 	4	&	5 	&	6	&	6 \\
	\hline
	$5^+$			& 	4	&	6 	&	6	&	6 \\
	\hline
	\end{tabular}
	\end{center}
	\caption{The star chromatic index of the Cartesian products of two paths $\chis{P_m \cartp P_n}$}
	\label{tbl:squaregrid}
\end{table}

As a corollary, we establish the lower bound of $6$ colors for the Cartesian products, 
where one factor is a cycle and the other is a path of length at least $2$.

\begin{corollary}
	\label{cor:lower}
	For every pair of integers $m$ and $n$, where $m \ge 3$ and $n \ge 3$, we have
	$$
		\chis{C_m \cartp P_n} \ge 6.
	$$
\end{corollary}

\begin{proof}
	We first note that the graph $C_3 \cartp P_2$ is one of the two known examples of simple bridgeless cubic graphs
	that have star chromatic index equal to $6$~\cite{LuzMocSot19}.
	Then, for $m = 3$, by Observation~\ref{obs:sub}, we have $\chis{C_3 \cartp P_n} \ge \chis{C_3 \cartp P_2} = 6$.

	If $m = 4$, we proceed by a contradiction. 
	Suppose that $\chis{C_4 \cartp P_n} \le 5$. 
	Then, by Lemma~\ref{lem:tiling}, $\chis{C_{4\ell} \cartp P_n} \le 5$, for any integer $\ell$,
	and hence also $\chis{P_5 \cartp P_{n}} \le 5$, a contradiction.
	Finally, if $m \ge 5$, then we have
	$\chis{C_m \cartp P_n} \ge \chis{P_5 \cartp P_3} = 6$ by Theorem~\ref{thm:PnPn} and Observation~\ref{obs:sub}.
\end{proof}


\subsection{Cartesian products of cycles}
\label{sub:CnCn}

Having the Cartesian products of paths resolved, the logical direction of research is consideration
of cylinders and toroidal grids, i.e., the Cartesian products of cycles and paths, and the Cartesian products of two cycles.
We begin by giving some results about the latter.

Corollary~\ref{cor:lower} implies that the Cartesian product of any two cycles will need at least $6$ 
colors for a star edge-coloring. 
On the other hand, as we will show in this subsection, 
the star chromatic index of the Cartesian product of two cycles is at most $7$.
We first investigate the Cartesian products of $C_3$ with another cycle.
\begin{theorem}
	\label{lem:C3Cn}
	For every integer $n$, where $n \ge 3$, we have
	$$
		\chis{C_3 \cartp C_n} = 
			\begin{cases} 
				6, &\text{if $n = 3k$,} \\
				7, &\text{otherwise}
			\end{cases}
	$$
\end{theorem}

\begin{proof}
	By Corollary~\ref{cor:lower}, $\chis{C_3 \cartp C_n} \ge 6$.
	Now suppose that $n = 3k$ for some integer $k \ge 1$.
	If $k=1$, then there is a star $6$-edge-coloring of $C_3 \cartp C_3$ 
	(one is depicted in Figure~\ref{fig:C3C3}).
	Next, by Lemma~\ref{lem:tiling}, we have $\chis{C_3 \cartp C_n} = 6$.
	
	Using Algorithm~\ref{alg:color}, we infer that the Cartesian product $C_3 \cartp P_3$ 
	has only one star $6$-edge-coloring up to a permutation of colors.
	Namely, three colors, say $0$, $1$, and $2$, appear on $C_3$-fibers,
	and the colors $4$, $5$, and $6$ on the $P_3$-fibers.
	Since $C_3 \cartp P_3$ is a subgraph of every graph $C_3 \cartp C_n$, 
	it follows that such Cartesian products admit a star $6$-edge-coloring only
	when $n$ is divisible by $3$.
	
	Therefore, if $n \neq 3k$ for every integer $k$, then $\chis{C_3 \cartp C_n} \ge 7$.
	In Figures~\ref{fig:C3C7} and~\ref{fig:C3C8}, a star $7$-edge-coloring of
	$C_3 \cartp C_7$ and $C_3 \cartp C_8$, respectively, is depicted.
	Observe that, in both colorings, a star $6$-edge-coloring of $C_3 \cartp C_3$ is included.
	Hence, by Lemma~\ref{lem:combine} and Theorem~\ref{thm:frob}, 
	we have $\chis{C_3 \cartp C_n} = 7$ for every $n$, $n \ge 7$, not divisible by $3$.
	The remaining two cases, namely $n=4$ and $n=5$, are depicted in Figures~\ref{fig:C3C4} and~\ref{fig:C3C5}, respectively.
\end{proof}

Similarly as in the proof of Theorem~\ref{lem:C3Cn}, 
we can use Lemma~\ref{lem:tiling} (twice)
to extend the star edge-coloring of $C_3 \cartp C_3$ to products of cycles of lengths divisible by $3$.
\begin{corollary}
	\label{cor:C3kC3k}
	For every pair of positive integers $k$ and $\ell$, we have
	$$
		\chis{C_{3k} \cartp C_{3\ell}} = 6\,.
	$$
\end{corollary}

We proceed with a result about the Cartesian products of $C_4$ with another cycle.
\begin{theorem}
	\label{lem:C4kC2n}
	For every pair of positive integers $k$ and $\ell$, where $k \ge 1$ and $\ell \ge 2$, we have
	$$
		\chis{C_{4k} \cartp C_{2\ell}} = 6\,.
	$$
\end{theorem}

\begin{proof}
	We use the star $6$-edge-coloring $\sigma_{10}$ of $C_4 \cartp C_{10}$ depicted in Figure~\ref{fig:C4C4C6}.
	Note that $\sigma_{10}$ includes a star $6$-edge-coloring $C_4 \cartp C_4$, 
	and a star $6$-edge-coloring $C_4 \cartp C_6$.
	Therefore, by Lemma~\ref{lem:combine} and Theorem~\ref{thm:frob}, 
	we have $\chis{C_{4} \cartp C_{2\ell}} = 6$ for every integer $\ell \ge 4$.
	Finally, we use Lemma~\ref{lem:tiling} to infer $\chis{C_{4k} \cartp C_{2\ell}} = 6$ for every integer $k$.
\end{proof}

\begin{proposition}
	\label{lem:C4Cm}
	For $n \in \set{5,7,9,11}$, we have
	$$
		\chis{C_{4} \cartp C_{n}} = 7\,.
	$$
\end{proposition}

\begin{proof}
	Using Algorithm~\ref{alg:color}, we established that $\chis{C_{4} \cartp C_{n}} > 6$ for every $n \in \set{5,7,9}$.
	The bounds $\chis{C_{4} \cartp C_{n}} = 7$ follow from the star $7$-edge-colorings depicted
	in Figures~\ref{fig:C4C5}, \ref{fig:C4C7}, and~\ref{fig:C4C9}.
	
	In the case of $n = 11$, we split the computation in two steps. 
	First, using Algorithm~\ref{alg:color}, we determined that if the edges of some $C_{11}$-fiber 
	are colored in such a way that a same color appears twice on some $4$-path,
	then the coloring cannot be extended to a star $6$-edge-coloring of $C_4 \cartp C_{11}$.
	In the second step, the algorithm checked only the colorings in which every $4$-path in each $C_{11}$-fiber 
	had four colors on its edges.
	It turned out that such a coloring does not exist. 
	Therefore,
	$\chis{C_{4} \cartp C_{11}} = 7$ by the star edge-coloring depicted in Figure~\ref{fig:C4C7} and Lemma~\ref{lem:tiling}.
\end{proof}

\begin{theorem}
	\label{lem:C4Cn}
	For any odd integer $n$, where $n \ge 13$, we have
	$$
		\chis{C_{4} \cartp C_{n}} \le 7\,.
	$$
\end{theorem}

\begin{proof}
	In Figure~\ref{fig:C4C7}, we present a star $7$-edge-coloring of $C_4 \cartp C_7$ with a star $7$-edge-coloring of $C_4 \cartp C_4$ included.
	Thus, by Lemma~\ref{lem:combine} and Theorem~\ref{thm:frob}, we infer that $\chis{C_{4} \cartp C_{n}} \le 7$ for every odd $n$, where $n > 18$.
	Colorings for $n \in \set{13,15,17}$ can be obtained by using Lemma~\ref{lem:combine} and the colorings 
	depicted in Figures~\ref{fig:C4C5} (for $n = 15$) and~\ref{fig:C4C9} (for $n \in \set{13,17}$).
\end{proof}

\begin{theorem}
	\label{lem:C5Cn}
	For every integer $n$, where $n \ge 3$, we have
	$$
		\chis{C_5 \cartp C_{n}} = 7\,.
	$$
\end{theorem}

\begin{proof}
	The lower bounds $\chis{C_5 \cartp C_{n}} > 6$ for $3 \le n \le 6$, were established using 
	Algorithm~\ref{alg:color}.
	For $n \ge 7$, using Algorithm~\ref{alg:color}, we infer that $\chis{C_5 \cartp P_n} \ge 7$.
	Therefore, by Observation~\ref{obs:sub}, we have $\chis{C_5 \cartp C_{n}} \ge 7$.

	Star $7$-edge-colorings of $C_5 \cartp C_m$, for $m \in \set{3,4,5,7,11}$, are depicted in Figures~\ref{fig:C3C5},
	\ref{fig:C4C5}, \ref{fig:C5C5}, \ref{fig:C5C7}, and~\ref{fig:C5C11}, respectively.
	By Lemma~\ref{lem:tiling}, we also infer star $7$-edge-colorings of $C_5 \cartp C_m$ for $m \in \set{6,8,9,10}$.
	Finally, since in Figure~\ref{fig:C5C7}, a star $7$-edge-coloring of $C_5 \cartp C_3$ is included, 
	by Lemma~\ref{lem:combine} and Theorem~\ref{thm:frob}, we obtain 
	$\chis{C_{5} \cartp C_{n}} = 7$ for every integer $n \ge 12$.
\end{proof}

\begin{theorem}
	\label{lem:C6Cn}
	For every integer $n$, where $n \ge 3$, we have
	$$
		\chis{C_6 \cartp C_n} = 
			\begin{cases} 
				6, &\text{if $n \equiv 0 \bmod{3}$ or $n \equiv 0 \bmod{4}$,} \\
				7, &\text{otherwise.}
			\end{cases}
	$$
\end{theorem}

\begin{proof}
	By the star $6$-edge-colorings depicted in Figures~\ref{fig:C3C3} and~\ref{fig:C4C6},
	and by Lemma~\ref{lem:tiling}, we have $\chis{C_6 \cartp C_n} = 6$,
	for every integer $n$ divisible by $3$ or $4$.

	Now we show that $\chis{C_6 \cartp C_n} > 6$ if $n$ is not divisible by $3$ or $4$.
	First, we consider the graph $C_6 \cartp P_{31}$, where $P_{31}=v_1\dots v_{31}$.
	We start with a precolored $C_6$-fiber at the vertex $v_{16}$ (i.e., the middle $C_6$-fiber) 
	using each of the nine possible star $6$-edge-colorings of $C_6$ (up to symmetries and permutations of colors). 
	Using Algorithm~\ref{alg:color},
	we tried to extend such a precoloring to the whole $C_6 \cartp P_{31}$.
	For five colorings of the $C_6$-fiber, namely for
	$(0,1,0,2,0,3)$,
	$(0,1,0,2,1,2)$,
	$(0,1,0,2,1,3)$,
	$(0,1,2,0,1,3)$, and
	$(0,1,2,0,3,4)$,
	we obtain that such precolorings cannot be extended.

	For the remaining four precolorings, namely
	$(0,1,0,2,3,2)$,
	$(0,1,0,2,3,4)$,
	$(0,1,2,0,1,2)$, and
	$(0,1,2,3,4,5)$,
	we obtain 27\,078 colorings of $C_6 \cartp P_{31}$ in total.
	Some of them are either $4$-, or $6$-periodical, i.e., the initial coloring repeats on every $4$-th or $6$-th fiber,
	except at the final three fibers on both sides, where the coloring restrictions are relaxed.

	The remaining 26\,448 colorings correspond to the precoloring $(0,1,2,3,4,5)$, 
	and moreover, all $C_6$-fibers are colored by shifts of this precoloring, 
	and every pair of adjacent $C_6$-fibers is either colored with the same sequence of colors, or the coloring of one is the coloring of the other shifted by $1$. 
	In a more detailed analysis of these colorings, we find that, if they are periodic, then the period must be a multiple of 6.

	Thus, for the graphs $C_6 \cartp C_n$, it follows that they are star $6$-edge-colorable
	if $n$ is divisible by $3$ or $4$. Otherwise they are not star $6$-edge-colorable.
\end{proof}

\begin{theorem}
	\label{lem:C7Cn}
	For every integer $n$, where $n \ge 3$, we have
	$$
		\chis{C_7 \cartp C_{n}} = 7\,.
	$$
\end{theorem}

\begin{proof}
	The lower bounds $\chis{C_7 \cartp C_{n}} > 6$, for $n \in \set{3,4,5,6}$, 
	were established using Algorithm~\ref{alg:color}.
	For $n \ge 7$, using Algorithm~\ref{alg:color}, we infer that $\chis{C_7 \cartp P_n} \ge 7$.
	Therefore, by Observation~\ref{obs:sub}, we have $\chis{C_7 \cartp C_{n}} \ge 7$.

	Star $7$-edge-colorings of $C_7 \cartp C_n$, for $n \in \set{3,4,5,7}$, are depicted in Figures~\ref{fig:C3C7},
	\ref{fig:C4C7}, \ref{fig:C5C7}, and~\ref{fig:C7C7}, respectively.
	By Lemma~\ref{lem:tiling}, from these colorings, we also infer star $7$-edge-colorings of $C_7 \cartp C_n$ for $n \in \set{6,8,9,10}$.
	Moreover, in the coloring depicted in Figure~\ref{fig:C11C11}, a star $7$-edge-coloring of $C_7 \cartp C_{11}$ is included,
	and hence we also have $\chis{C_7 \cartp C_{11}} = 7$.	
	Finally, since in the coloring depicted in Figure~\ref{fig:C7C7}, a star $7$-edge-coloring of $C_7 \cartp C_3$ is included, 
	by Lemma~\ref{lem:combine} and Theorem~\ref{thm:frob}, we have 
	$\chis{C_{7} \cartp C_{n}} = 7$ for every $n \ge 12$.
\end{proof}

\begin{proposition}
	\label{lem:C8C9}
	For $C_8 \cartp C_{9}$, we have
	$$
		\chis{C_8 \cartp C_9} = 7\,.
	$$
\end{proposition}

\begin{proof}	
	We determined that $\chis{C_8 \cartp C_9} \ge 7$ by exhaustive computer search.
	Namely, we generated all $147$ non-isomorphic star $6$-edge-colorings of $C_9$ 
	and tried to extend each of them to the graph $C_8 \cartp C_9$. 
	None of them could be extended, thus $\chis{C_8 \cartp C_9} \ge 7$.
	The equality follows from Lemma~\ref{lem:tiling} and the fact that $\chis{C_4 \cartp C_9} = 7$.
\end{proof}

Finally, we give a general result, showing that $7$ is the upper bound for the star chromatic index of 
the Cartesian products of any two cycles.
\begin{theorem}
	\label{lem:CmCn}
	For every pair of positive integers $m$ and $n$, where $3 \le m \le n$, we have
	$$
		\chis{C_m \cartp C_n} \le 7\,.
	$$
\end{theorem}

\begin{proof}
	By Theorems~\ref{lem:C3Cn}-\ref{lem:C7Cn}, we have
	$\chis{C_m \cartp C_n} \le 7$ for $3 \le m \le 7$ and $n \ge 3$.
	Furthermore, by Lemma~\ref{lem:tiling}, 
	we can use Theorem~\ref{lem:C4Cn} to obtain a star $7$-edge-coloring of $C_8 \cartp C_n$,
	Theorem~\ref{lem:C3Cn} to obtain a star $7$-edge-coloring of $C_9 \cartp C_n$,
	and Theorem~\ref{lem:C5Cn} to obtain a star $7$-edge-coloring of $C_{10} \cartp C_n$, 
	for every $n \ge 3$. 
	The star $7$-edge-coloring of $C_{11} \cartp C_{11}$ is depicted in Figure~\ref{fig:C11C11}.
	
	We complete the proof by showing that $\chis{C_m \cartp C_n} \le 7$ if $m,n \ge 12$.	
	Note that the star $7$-edge-coloring of $C_7 \cartp C_7$ depicted in Figure~\ref{fig:C7C7},
	includes a $7$-edge-coloring of $C_3 \cartp C_7$ and a $7$-edge-coloring of $C_7 \cartp C_3$.
	Furthermore, the latter two colorings include a common star $7$-edge-coloring of $C_3 \cartp C_3$. 
	This fact enables us to use Lemma~\ref{lem:combine} and Theorem~\ref{thm:frob} to obtain $\chis{C_m \cartp C_n} \le 7$ for $m,n \ge 12$.
\end{proof}

The above results are summarized in Table~\ref{tbl:toroidal}.
\begin{table}[ht] 
\begin{center}
\begin{tabular}{|c||C{\colwidth}|C{\colwidth}|C{\colwidth}|C{\colwidth}|C{\colwidth}|C{\colwidth}|C{\colwidth}|C{\colwidth}|C{\colwidth}|C{\colwidth}|}
\hline
$m \backslash n$ & 	3 	& 	4 	& 	5	&	6	&	7	&	8	&	9	&	10	&	11	&	12 \\
\hline
\hline
3				& 	6	&	7 	&	7	&	6	&	7	&	7	&	6	&	7 	&	7 	&	6 \\
\hline
4				& 	7	&	6 	&	7	&	6	&	7	&	6	&	7   & 	6	&	7 	&	6 \\
\hline
5				& 	7	&	7 	&	7	&	7	&	7	&	7	&	7 	&	7	&	7 	&	7 \\
\hline
6				& 	6	&	6 	&	7	&	6	&	7	&	6	&	6 	&	7 	&	7 	&	6 \\
\hline
7				& 	7	&	7 	&	7	&	7	&	7	&	7	&	7 	&	7	&	7 	&	7 \\
\hline
8				& 	7	&	6 	&	7	&	6	&	7	&	6	&	7 	&	6 	&	\textcolor{red}{$7^-$} 	&	6 \\
\hline
9				& 	6	&	7 	&	7	&	6	&	7	&	7	&	6 	& 	\textcolor{red}{$7^-$}	&	\textcolor{red}{$7^-$}	&	6 \\
\hline
10				& 	7	&	6 	&	7	&	7	& 	7	&	6	&	\textcolor{red}{$7^-$} 	& 	\textcolor{red}{$7^-$}	&	\textcolor{red}{$7^-$}	&	6 \\
\hline
11				& 	7	&	7 	&	7	&	7	& 	7	&	\textcolor{red}{$7^-$}	&	\textcolor{red}{$7^-$} 	& 	\textcolor{red}{$7^-$}	&	\textcolor{red}{$7^-$}	&	\textcolor{red}{$7^-$} \\
\hline
12				& 	6	&	6 	&	7	&	6	& 	7	&	6	&	6 	& 	6	&	\textcolor{red}{$7^-$}	&	6 \\
\hline
\end{tabular}
\end{center}
\caption{The star chromatic index of the Cartesian products of cycles $\chis{C_m \cartp C_n}$.
	In red, we denote the cases, where the exact bounds are not established yet. 
	The value $7^-$ means that the exact value of the star chromatic index is either $6$ or $7$.}
\label{tbl:toroidal}
\end{table}


\subsection{Cartesian products of cycles and paths}
\label{sub:CnPn}

In the last part of this section, we give results about the Cartesian products of paths and cycles.
We begin with proving the cases for specific lengths of cycles.

\begin{theorem}
	\label{lem:PnC2k}
	For every pair of integers $k$ and $n$, where $k \ge 2$ and $n \ge 3$, we have
	$$
		\chis{C_{2k} \cartp P_n} = 6\,.
	$$
\end{theorem}

\begin{proof}
	By Corollary~\ref{cor:lower}, we have $\chis{C_{2k} \cartp P_n} \ge 6$.
	On the other hand, $6$ colors also suffice by Theorem~\ref{lem:C4kC2n}, 
	since $C_{2k} \cartp P_n$ is a subgraph of $C_{2k} \cartp C_{4\ell}$ for every $\ell \ge n/4$.
\end{proof}

\begin{theorem}
	\label{lem:PnC3}
	For every pair of integers $k$ and $n$, where $n \ge 3$, we have
	$$
		\chis{C_{3k} \cartp P_n} = 6\,.
	$$
\end{theorem}

\begin{proof}
	By Corollary~\ref{cor:lower}, we have $\chis{C_{3k} \cartp P_n} \ge 6$.
	On the other hand, $6$ colors also suffice by Corollary~\ref{cor:C3kC3k}, 
	since $C_{3k} \cartp P_n$ is a subgraph of $C_{3k} \cartp C_{3\ell}$ for every $\ell \ge n/3$.
\end{proof}

\begin{theorem}
	\label{lem:PnC5}
	For every integer $n$, where $n \ge 3$, we have
	$$
		\chis{C_5 \cartp P_n} = 
			\begin{cases} 
				6, &\text{if $n \in \set{3,4,5,6}$,} \\
				7, &\text{if $n  \ge 7$.}
			\end{cases}
	$$
\end{theorem}

\begin{proof}	
	Let $G = C_5 \cartp P_n$ for some integer $n \ge 3$.
	Suppose first that $n \in \set{3,4,5,6}$.
	By Theorem~\ref{thm:PnPn}, we have $\chis{P_3 \cartp P_5} = 6$,
	and thus, since $P_{5} \cartp P_3$ is a subgraph of $G$,
	it follows that $\chis{G} \ge 6$.
	On the other hand, in Figure~\ref{fig:C5P6}, we give a star $6$-edge-coloring of $C_5 \cartp P_6$,
	hence establishing $\chis{C_5 \cartp P_n} = 6$ for every $n \in \set{3,4,5,6}$.
	Now, suppose that $n \ge 7$.
	Using Algorithm~\ref{alg:color},
	we infer that $\chis{C_5 \cartp P_n} \ge 7$.
	The upper bound $\chis{C_5 \cartp P_n} \le 7$ follows 
	from the fact that $\chis{C_5 \cartp C_{5k}} = 7$ for every positive integer $k$ (see Theorem~\ref{lem:C5Cn} and Figure~\ref{fig:C5C5}).	
\end{proof}

\begin{theorem}
	\label{lem:PnC7}
	For every integer $n$, where $n \ge 3$, we have
	$$
		\chis{C_7 \cartp P_n} = 
			\begin{cases} 
				6, &\text{if $n \in \set{3,4,5,6}$,} \\
				7, &\text{if $n  \ge 7$.}
			\end{cases}
	$$
\end{theorem}

\begin{proof}
	Let $G = C_7 \cartp P_n$ for some integer $n \ge 3$.
	Suppose first that $n \in \set{3,4,5,6}$.
	Since $\chis{P_3 \cartp P_5} = 6$, we again have $\chis{G} \ge 6$.
	On the other hand, in Figure~\ref{fig:C7P6}, we give a star $6$-edge-coloring of $G$, 
	and hence $\chis{G} = 6$ for every $n \in \set{3,4,5,6}$.
	Suppose now that $n \ge 7$.
	Using Algorithm~\ref{alg:color}, we infer that $\chis{C_7 \cartp P_7} = 7$, and hence $\chis{G} \ge 7$.
	The equality is established by Figure~\ref{fig:C3C7}, where a pattern for a star $7$-edge-coloring of $C_7 \cartp C_{3k}$
	is presented. Since $G$ is a subgraph of $C_7 \cartp C_{3k}$, for $k$ large enough, the statement follows.
\end{proof}

We now turn our attention to the Cartesian products of cycles and paths $P_m$, for $m \in \set{2,3,4}$.

\begin{theorem}
	\label{lem:P2Cn}
	For every integer $m$, where $m \ge 3$, we have
	$$
		\chis{C_m \cartp P_2} = 
			\begin{cases} 
				6, &\text{if $m = 3$,} \\
				4, &\text{if $m \equiv 0 \bmod{4}$,} \\
				5, &\text{otherwise.}
			\end{cases}
	$$
\end{theorem}

\begin{proof}
	For $m=3$, recall that $\chis{C_3 \cartp P_2} = 6$.
	If $m \equiv 0 \bmod{4}$, then the graph $C_m \cartp P_2$ covers the graph $Q_3$, 
	and hence its star chromatic index equals $4$ by Theorem~\ref{thm:cubic}.
	Finally, if $m \not\equiv 0 \bmod{4}$, by Theorem~\ref{thm:cubic}, we have
	$\chis{C_m \cartp P_2} \ge 5$. The equality follows from Theorem~7 in~\cite{OmoDas18}.
\end{proof}

\begin{theorem}
	\label{lem:PnCn}
	For every pair of integers $m$ and $n$, where $m \ge 3$ and $n \in \set{3,4,5,6}$, we have
	$$
		\chis{C_m \cartp P_n} = 6\,.
	$$
\end{theorem}

\begin{proof}
	First, recall that $6$ colors are needed in all the cases by Corollary~\ref{cor:lower}.
	Next, since $C_m \cartp P_6$ contains all the graphs $C_m \cartp P_n$, for $n \in \set{3,4,5}$,
	it suffices to show that there exists a star $6$-edge-coloring of $C_m \cartp P_6$.
	
	By Lemma~\ref{lem:tiling}, Theorems~\ref{lem:C3Cn} and~\ref{lem:C4kC2n}, 
	we have $\chis{C_{3k} \cartp P_6} = \chis{C_{4k} \cartp P_6} = 6$,
	for any positive integer $k$.
	Similarly, the star $6$-edge-colorings depicted in Figures~\ref{fig:C5P6} and~\ref{fig:C7P6} together with Lemma~\ref{lem:tiling}
	imply $\chis{C_{5k} \cartp P_6} = \chis{C_{7k} \cartp P_6} = 6$.
	Furthermore, using the star $6$-edge-coloring of $C_{10} \cartp P_6$ depicted in Figure~\ref{fig:P6C10},
	which includes a star $6$-edge-coloring of $C_3 \cartp P_6$, we obtain
	$\chis{C_{m} \cartp P_6} = 6$ for every $m \in \set{10,13,16}$. 
	Moreover, together with Lemma~\ref{lem:combine} and Theorem~\ref{thm:frob}, this coloring implies $\chis{C_{m} \cartp P_6} = 6$ for every $m \ge 18$.
	For the remaining two cases, namely $m \in \set{11,17}$, the corresponding colorings are depicted
	in Figures~\ref{fig:P8C11} and \ref{fig:P8C14}
	(note that, in fact, we have even more, namely, we give a star $6$-edge-coloring of $C_{11} \cartp P_8$ 
	and a star $6$-edge-coloring of $C_{14} \cartp P_8$ which includes a star $6$-edge-coloring of $C_{3} \cartp P_8$).
\end{proof}

We collect known and our new results in Table~\ref{tbl:cylinders}.

\begin{table}[htb!] 
\begin{center}
\begin{tabular}{|c||C{\colwidth}|C{\colwidth}|C{\colwidth}|C{\colwidth}|C{\colwidth}|C{\colwidth}|C{\colwidth}|C{\colwidth}|}
\hline
$m \backslash n$ & 	2 	& 	3 	& 	4	&	5	&	6	&	7	&	8	&	$9^+$ \\
\hline
\hline
3				& 	6	&	6 	&	6	&	6 	& 	6	& 	6	& 	6 	& 	6 \\
\hline
4				& 	4	&	6 	&	6	&	6 	& 	6	& 	6	& 	6 	& 	6 \\
\hline
5				& 	5	&	6 	&	6	&	6 	&	6 	&	7 	& 	7 	& 	7 \\
\hline
6				& 	5	&	6 	&	6	&	6 	& 	6	& 	6	& 	6 	& 	6 \\
\hline
7				& 	5	&	6 	&	6	&	6 	& 	6 	& 	7	& 	7 	& 	7 \\
\hline
8				& 	4	&	6 	&	6	&	6 	& 	6 	& 	6	& 	6 	& 	6 \\
\hline
9				& 	5	&	6 	&	6	&	6 	& 	6  	&	6	&	6 	& 	6 \\
\hline
10				& 	5	&	6 	&	6	&	6 	& 	6	&	6 	& 	6 	& 	6 \\
\hline
11				& 	5	&	6 	&	6	&	6 	& 	6 	&	6	&	6  	& 	\textcolor{red}{$7^-$} \\
\hline
12				& 	4	&	6 	&	6	&	6 	& 	6 	& 	6 	& 	6 	& 	6 \\
\hline
13				& 	5	&	6 	&	6	&	6	& 	6 	&	\textcolor{red}{$7^-$}	&	\textcolor{red}{$7^-$}  	& 	\textcolor{red}{$7^-$} \\
\hline
14				& 	5	&	6 	&	6	&	6 	& 	6 	&	6	&	6 	& 	6 \\
\hline
15				& 	5	&	6 	&	6	&	6 	& 	6	& 	6	& 	6 	& 	6 \\
\hline
16				& 	4	&	6 	&	6	&	6 	& 	6	& 	6	& 	6 	& 	6 \\
\hline
17				& 	5	&	6 	&	6	&	6 	& 	6 	&	6	&	6 	& 	\textcolor{red}{$7^-$} \\
\hline
18				& 	5	&	6 	&	6	&	6 	& 	6	& 	6	& 	6 	& 	6 \\
\hline
$19^+$ and $m \equiv 0 \bmod{4}$ 
				& 	4	&	6 	&	6	&	6	& 	6 	&	6	& 	6 	& 	6 \\
\hline
$19^+$ and $m \equiv r \bmod{12}$, $r \in \set{2,3,6,9,10}$ 
				& 	5	&	6 	&	6	&	6	& 	6 	&	6	&	6 	& 	6 \\
\hline
$19^+$ and $m \equiv r \bmod{12}$, $r \in \set{1,5,7,11}$ 
				& 	5	&	6 	&	6	&	6	& 	6 	&	\textcolor{red}{$7^-$}	&	\textcolor{red}{$7^-$}  	& 	\textcolor{red}{$7^-$} \\				
\hline

\end{tabular}
\end{center}
\caption{The star chromatic index of the Cartesian products of cycles and paths $\chis{C_m \cartp P_n}$. 
	In red, we denote the cases, where the exact bounds are not established yet.
	The value $7^-$ means that the exact value of the star chromatic index is either $6$ or $7$.}
\label{tbl:cylinders}
\end{table}


\section{Conclusion}
\label{sec:conc}

In this paper, we have established tight upper bounds for the Cartesian product of cycles and paths,
and the Cartesian products of two cycles. 
We proved that $7$ colors always suffice for star edge-colorings of these graphs, 
which is, in a way, a surprising bound, especially because at least $6$ colors are always needed
as soon as one of the factors is not isomorphic to $P_2$.
Although we improved existing bounds and proved a number of exact values, 
there are still some open questions.
We are very confident that the following conjecture is true.
\begin{conjecture}
	\label{conj:PnCn}
	There exist constants $K_1$ and $K_2$ such that for every pair of integers $m$ and $n$, where $m \ge K_1$ and $n \ge K_2$,
	we have
	$$
		\chis{C_m \cartp P_n} = 6\,.
	$$
\end{conjecture}

Note that Conjecture~\ref{conj:PnCn} is equivalent to the following.
\begin{conjecture}
	\label{conj:PnCn'}
	There exists a constant $K$ such that for every integer $m$, where $m \ge K$, 
	there exists an integer $n$ such that
	$$
		\chis{C_m \cartp C_n} = 6\,.
	$$
\end{conjecture}

In fact, we believe (with a bit lower confidence) that the following stronger version of Conjecture~\ref{conj:PnCn} can also be confirmed.
\begin{conjecture}
	\label{conj:CnCn}
	There exists a constant $L$ such that for every pair of integers $m$ and $n$, where $m,n \ge L$,
	we have
	$$
		\chis{C_m \cartp C_n} = 6\,.
	$$
\end{conjecture}

The above conjectures seem to be challenging, since we were not able to observe any straightforward pattern in colorings of 
different Cartesian products, despite the help of computer in our constructions.

There are also some (maybe) less complicated open questions regarding the Cartesian product of the cycles $C_m$, $m \in \set{11,13,17}$,
with paths of arbitrary lengths.
\begin{question}
	What is the star chromatic index of $C_m \cartp P_n$ for $m \in \set{11,13,17}$ and $n \ge 7$?
\end{question}
As we observed in Theorems~\ref{lem:PnC5} and~\ref{lem:PnC7}, the star chromatic index increases 
for the Cartesian products of the cycles $C_5$ and $C_7$ with paths on at least $7$ vertices.
We expect this phenomenon will repeat also for some cycle $C_m$, where $m \in \set{11,13,17}$,
although, in the case of $C_{11}$ and $C_{17}$, we found star $6$-edge-colorings of $C_{11} \cartp P_8$ and $C_{17} \cartp P_8$, 
which could indicate that $6$ colors are sufficient in these cases.

To conclude, as in the case of complete graphs, also for the Cartesian products of paths and cycles 
(and two cycles), it is hard to determine their chromatic index with our current methods. 
However, it seems that the latter will be easier to resolve as the former.

\paragraph{Acknowledgment.} 

PH and MM acknowledge the financial support from the project GA20--09525S of the Czech Science Foundation.
BL was partly supported by the Slovenian Research Agency Program P1--0383 and the project J1--1692.
EM and RS were supported by the Slovak Research and Development Agency under the Contract No. APVV--15--0116 and by the Slovak VEGA Grant 1/0368/16.

\bibliographystyle{plain}
{\small
	\bibliography{MainBase}

\begin{thebibliography}{10}

\bibitem{BezLuzMocSotSkr16}
{L'}. Bezegov\'{a}, B.~Lu\v{z}ar, M.~Mockov\v{c}iakov\'{a}, R.~Sot\'{a}k, and
  R.~\v{S}krekovski.
\newblock Star edge coloring of some classes of graphs.
\newblock {\em J. Graph Theory}, 81:73--82, 2016.

\bibitem{CasGraRas19}
C.~J. Casselgren, J.~B. Granholm, and A.~Raspaud.
\newblock On star edge colorings of bipartite and subcubic graphs.
\newblock 2019.
\newblock ArXiv Preprint (\url{https://arxiv.org/pdf/1912.02467.pdf}).

\bibitem{DvoMohSam13}
Z.~Dvo\v{r}\'{a}k, B.~Mohar, and R.~\v{S}\'{a}mal.
\newblock Star chromatic index.
\newblock {\em J. Graph Theory}, 72:313--326, 2013.

\bibitem{Gru73}
B.~Gr{\" u}nbaum.
\newblock Acyclic coloring of planar graphs.
\newblock {\em Israel J. Math.}, 14:390--412, 1973.

\bibitem{KerRas18}
S.~Kerdjoudj and A.~Raspaud.
\newblock List star edge coloring of sparse graphs.
\newblock {\em Discrete Appl. Math.}, 238:115--125, 2018.

\bibitem{LeiShiSon18}
H.~Lei, Y.~Shi, and Z.-X. Song.
\newblock Star chromatic index of subcubic multigraphs.
\newblock {\em J. Graph Theory}, 88:566--576, 2018.

\bibitem{LeiShiSonWan18}
H.~Lei, Y.~Shi, Z.-X. Song, and T.~Wang.
\newblock Star $5$-edge-colorings of subcubic multigraphs.
\newblock {\em Discrete Math.}, 341:950--956, 2018.

\bibitem{LiuDen08}
X.-S. Liu and K.~Deng.
\newblock An upper bound on the star chromatic index of graphs with {$\Delta
  \ge 7$}.
\newblock {\em J. Lanzhou Univ. (Nat. Sci.)}, 44:94--95, 2008.

\bibitem{LuzMocSot19}
B.~Lu\v{z}ar, M.~Mockov\v{c}iakov\'{a}, and R.~Sot\'{a}k.
\newblock Note on list star edge-coloring of subcubic graphs.
\newblock {\em J. Graph Theory}, 90:304--310, 2019.

\bibitem{Moc13}
M.~Mockov\v{c}iakov\'{a}.
\newblock {\em {Distance Constrained Edge Colorings of Graphs}}.
\newblock PhD thesis, P. J. \v{S}af\'{a}rik University, Faculty of Science,
  Ko\v{s}ice, 2013.

\bibitem{OmoDas18}
B.~Omoomi and M.~V. Dastjerdi.
\newblock {Star Edge Coloring of the Cartesian Product of Graphs}.
\newblock 2018.
\newblock Preprint (\url{https://arxiv.org/abs/1802.01300}).

\bibitem{Syl82}
J.~J. Sylvester.
\newblock {On Subvariants, i.e. Semi-Invariants to Binary Quantics of an
  Unlimited Order}.
\newblock {\em Amer. J. Math.}, 5:79--136, 1882.

\bibitem{WanWanWan18}
Y.~Wang, W.~Wang, and Y.~Wang.
\newblock Edge-partition and star chromatic index.
\newblock {\em Appl. Math. Comput.}, 333:480--489, 2018.

\bibitem{WanWanWan19}
Y.~Wang, Y.~Wang, and W.~Wang.
\newblock Star edge-coloring of graphs with maximum degree four.
\newblock {\em Appl. Math. Comput.}, 340:268--275, 2019.

\end{thebibliography}
}

\newpage

\appendix

\section{Cartesian products of two cycles}

\begin{figure}[ht]
\caption{Cartesian products of $C_3$ with cycles}

	\begin{subfigure}[t]{.5\textwidth}
	$$
		\includegraphics[scale=\scaleConst]{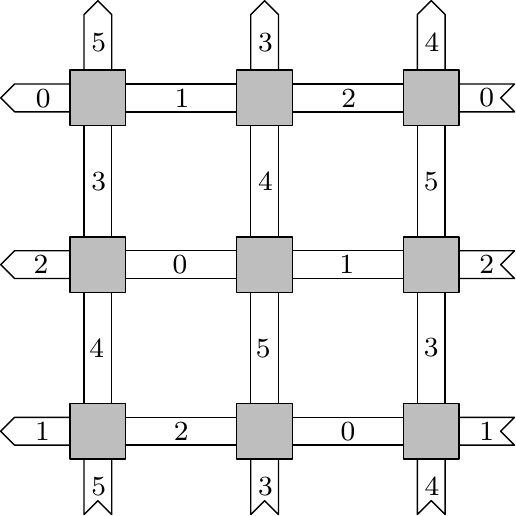}
	$$
		\caption{A star $6$-edge-coloring of $C_3 \cartp C_3$}
		\label{fig:C3C3}
	\end{subfigure}	
	\begin{subfigure}[t]{.4\textwidth}
	$$
		\includegraphics[scale=\scaleConst]{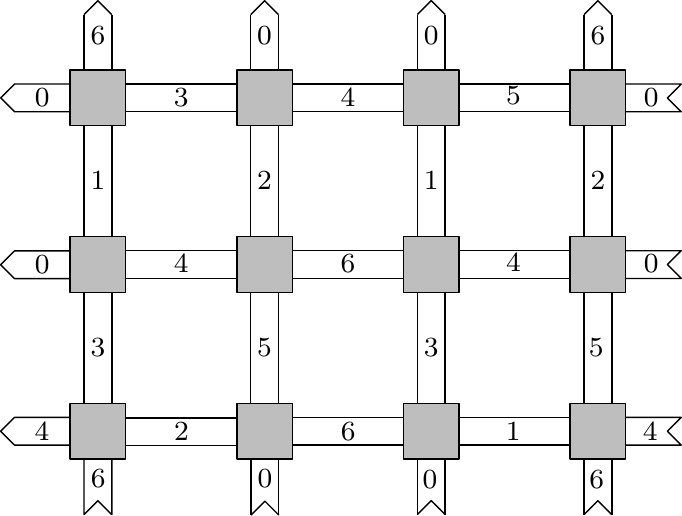}
	$$
		\caption{A star $7$-edge-coloring of $C_3 \cartp C_4$}
		\label{fig:C3C4}
	\end{subfigure}	
	
	\vspace{0,5cm}					
	\begin{subfigure}[t]{.45\textwidth}
	$$	
		\includegraphics[scale=\scaleConst]{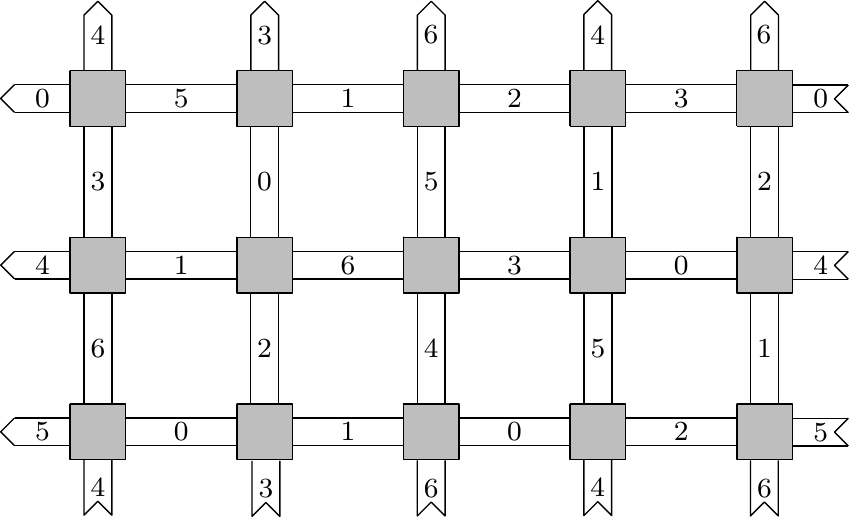}
	$$
		\caption{A star $7$-edge-coloring of $C_3 \cartp C_5$}
		\label{fig:C3C5}
	\end{subfigure}
	\begin{subfigure}[t]{.55\textwidth}
	$$
		\includegraphics[scale=\scaleConst]{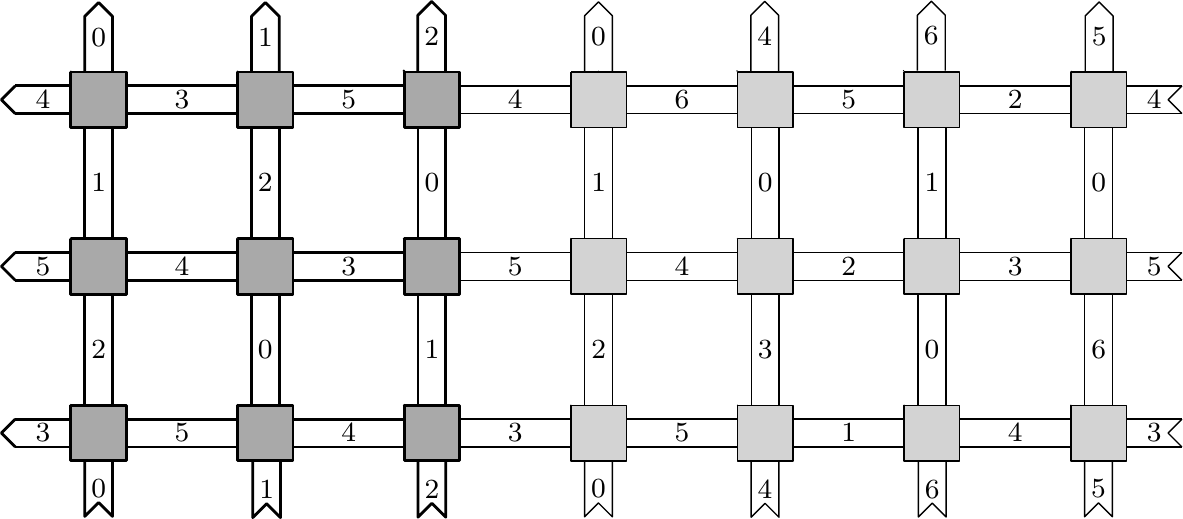}
	$$
		\caption{A star $7$-edge-coloring of $C_3 \cartp C_7$
			including a star $6$-edge-coloring of $C_3 \cartp C_3$ (darker vertices)}
		\label{fig:C3C7}
	\end{subfigure}
	
	\vspace{0,5cm}
	\begin{subfigure}[t]{\textwidth}
	$$
		\includegraphics[scale=\scaleConst]{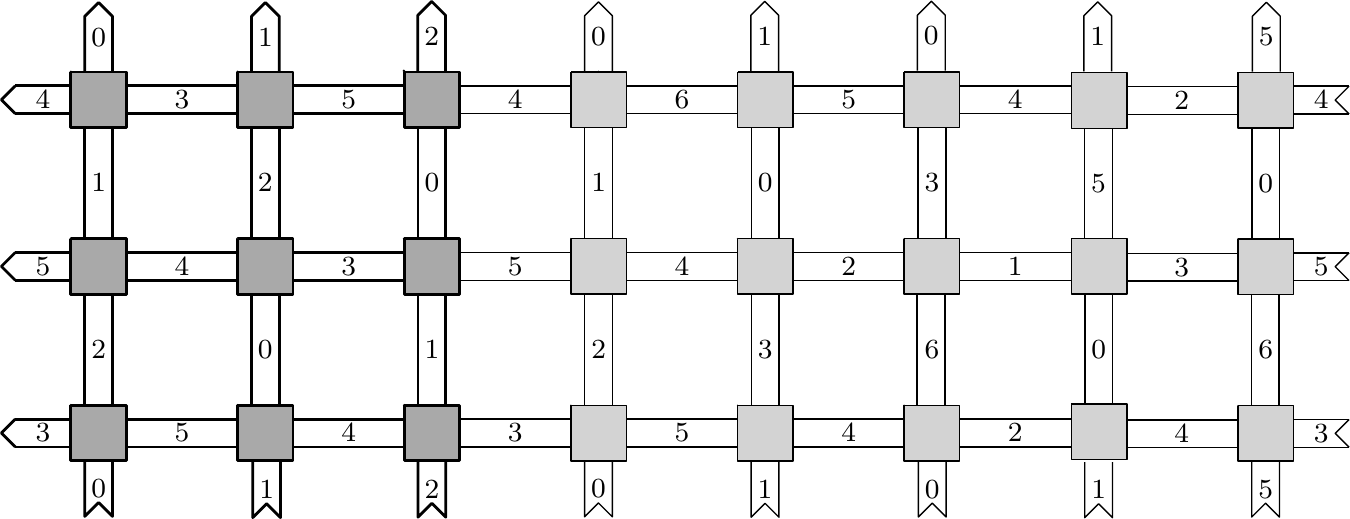}		
	$$
		\caption{A star $7$-edge-coloring of $C_3 \cartp C_8$
			including a star $6$-edge-coloring of $C_3 \cartp C_3$ (darker vertices)}
		\label{fig:C3C8}
	\end{subfigure}
	
\end{figure}

\begin{figure}[ht]
\caption{Cartesian products of $C_4$ with cycles}

	\begin{subfigure}[b]{.45\textwidth}
		$$
			\includegraphics[scale=\scaleConst]{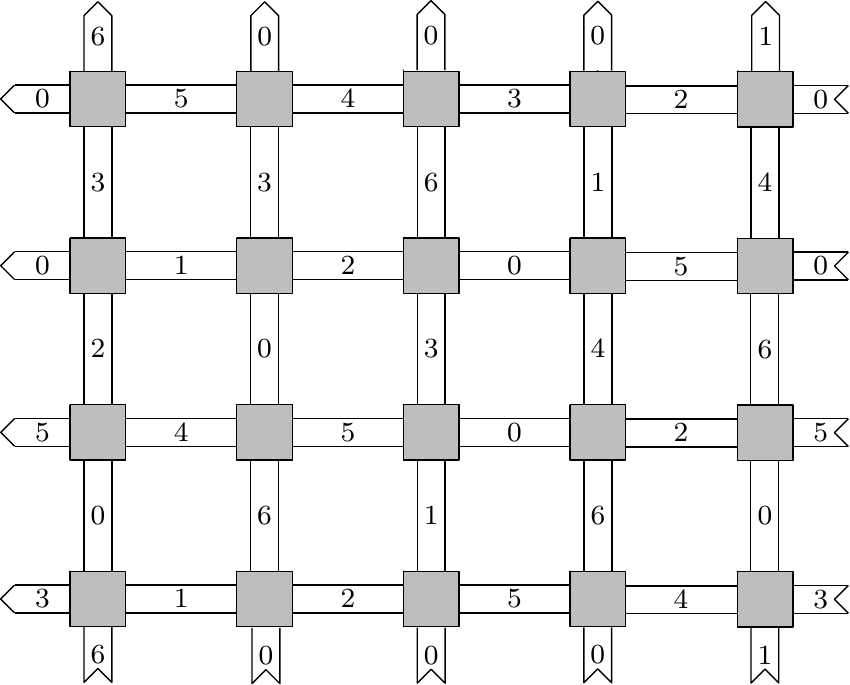}
		$$		
		\caption{A star $7$-edge-coloring of $C_4 \cartp C_5$}
		\label{fig:C4C5}
	\end{subfigure}	
	\begin{subfigure}[b]{.49\textwidth}
		$$
			\includegraphics[scale=\scaleConst]{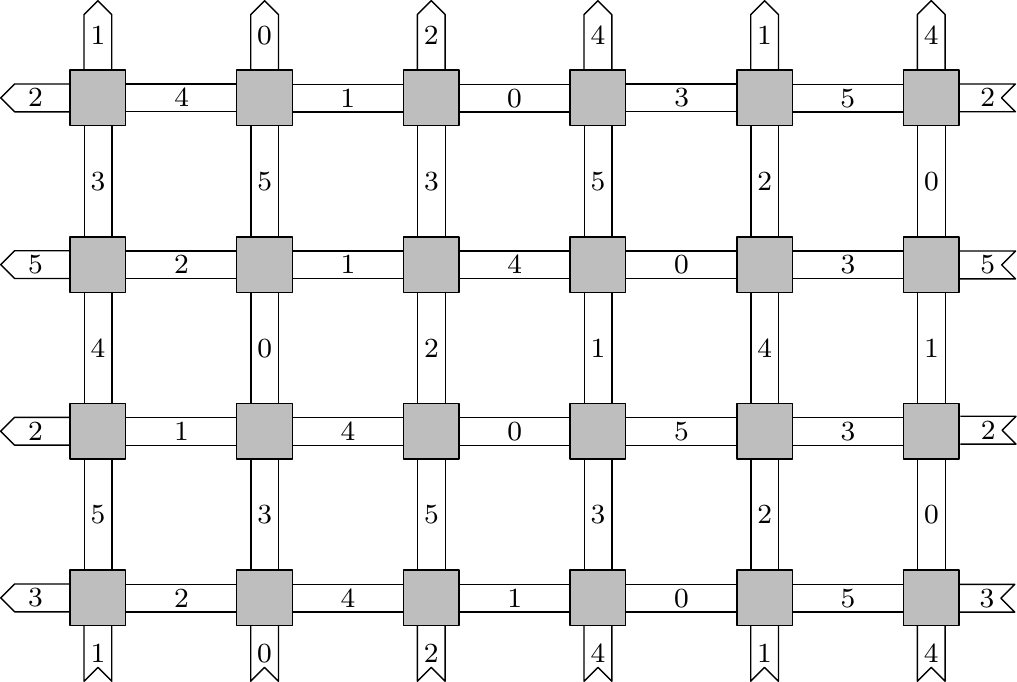}
		$$		
		\caption{A star $6$-edge-coloring of $C_4 \cartp C_6$}
		\label{fig:C4C6}
	\end{subfigure}
	
	\vspace{0.5cm}
	\begin{subfigure}[b]{\textwidth}
		$$
			\includegraphics[scale=\scaleConst]{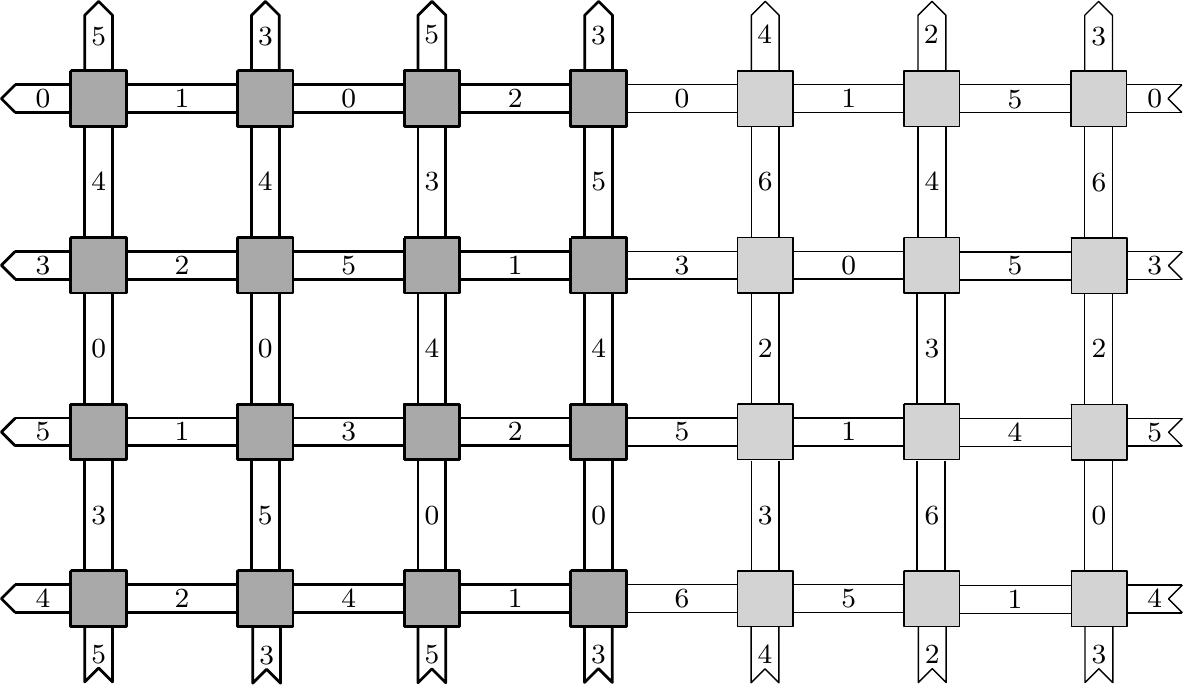}
		$$		
		\caption{A star $7$-edge-coloring of $C_4 \cartp C_7$ 
			including a star edge-coloring of $C_4 \cartp C_4$ (darker vertices)}
		\label{fig:C4C7}
	\end{subfigure}

	\vspace{0.5cm}
	\begin{subfigure}[b]{\textwidth}
		$$
			\includegraphics[scale=\scaleConst]{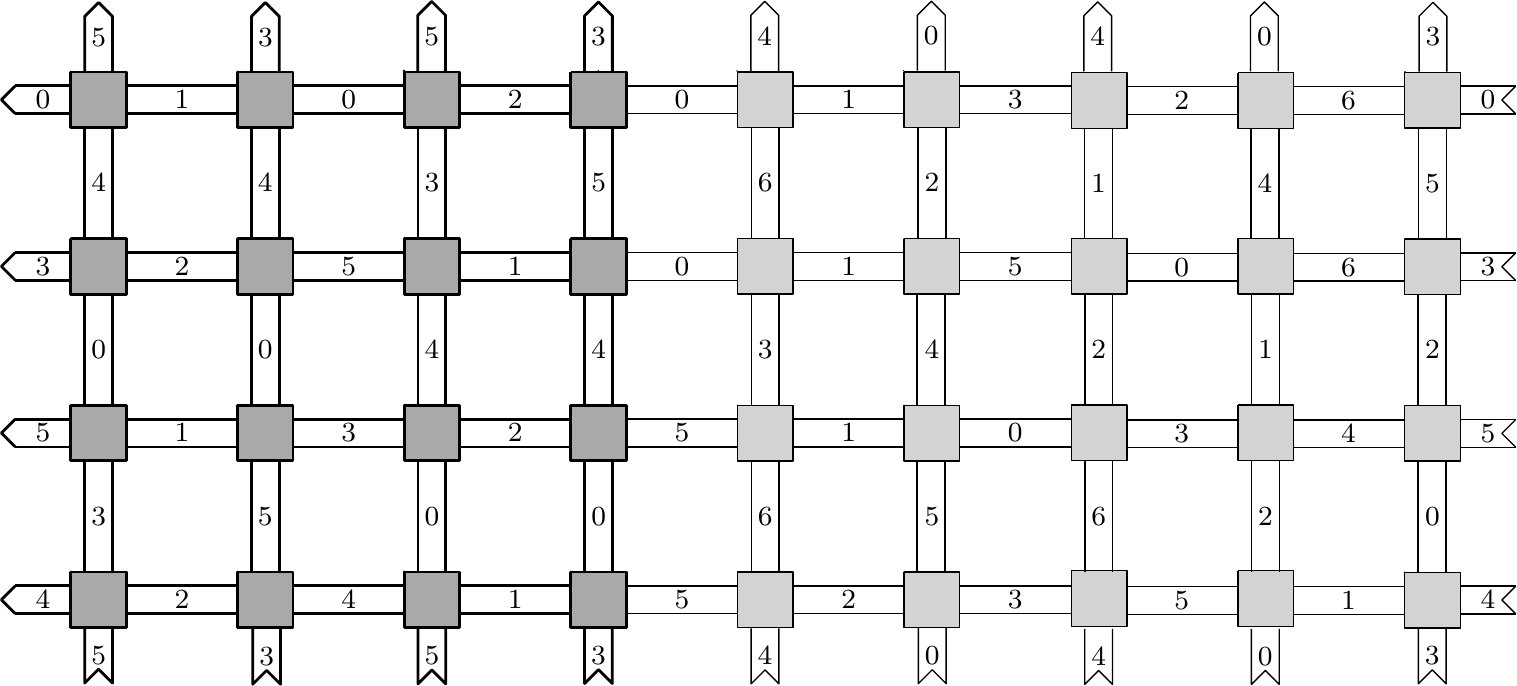}
		$$		
		\caption{A star $7$-edge-coloring of $C_4 \cartp C_9$
			including a star edge-coloring of $C_4 \cartp C_4$ (darker vertices)}
		\label{fig:C4C9}
	\end{subfigure}
	
\end{figure}

\begin{figure}[ht]
	\caption{A star $6$-edge-coloring of $C_4 \cartp C_{10}$ combined of star edge-colorings 
		of $C_4 \cartp C_4$ (lighter vertices) and $C_4 \cartp C_6$ (darker vertices)}
	$$
		\includegraphics[scale=\scaleConst]{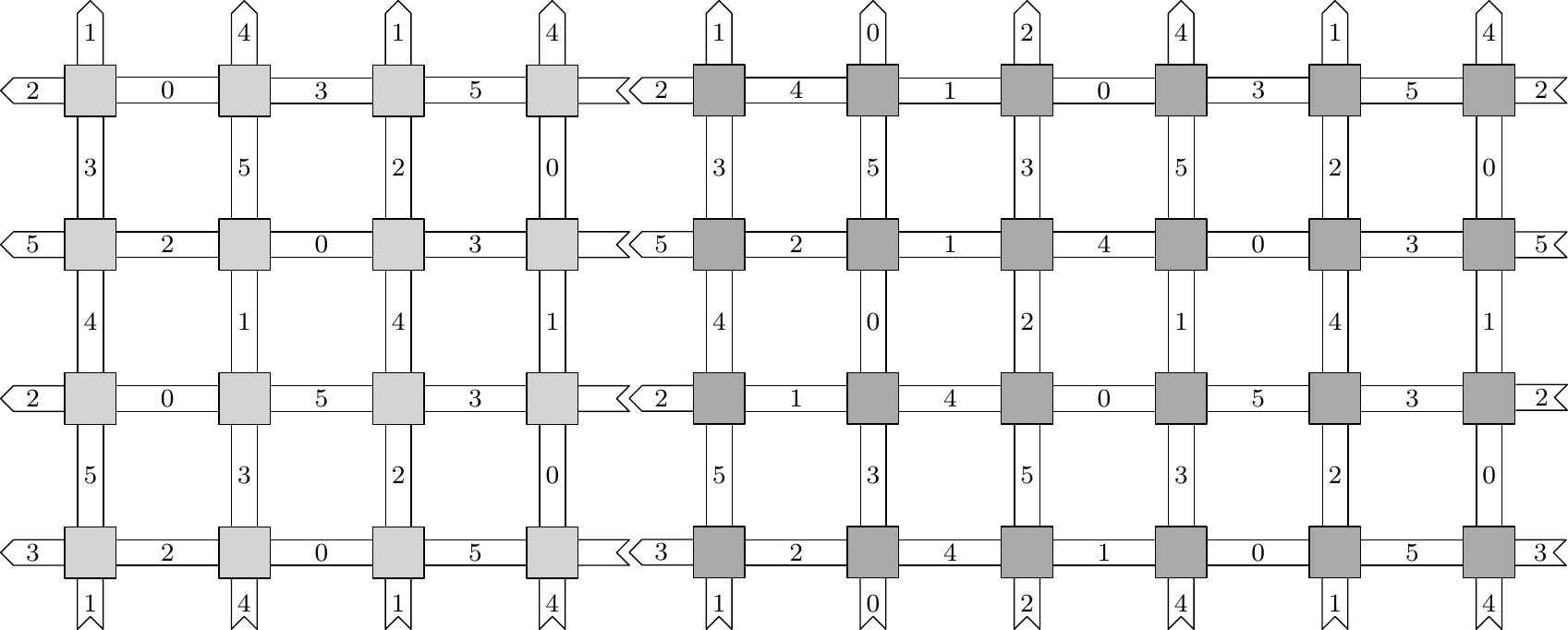}
	$$
	\label{fig:C4C4C6}
\end{figure}

\begin{figure}[ht]
	\caption{Cartesian products of $C_5$ with cycles}

	\begin{subfigure}[t]{.45\textwidth}
		$$
			\includegraphics[scale=\scaleConst]{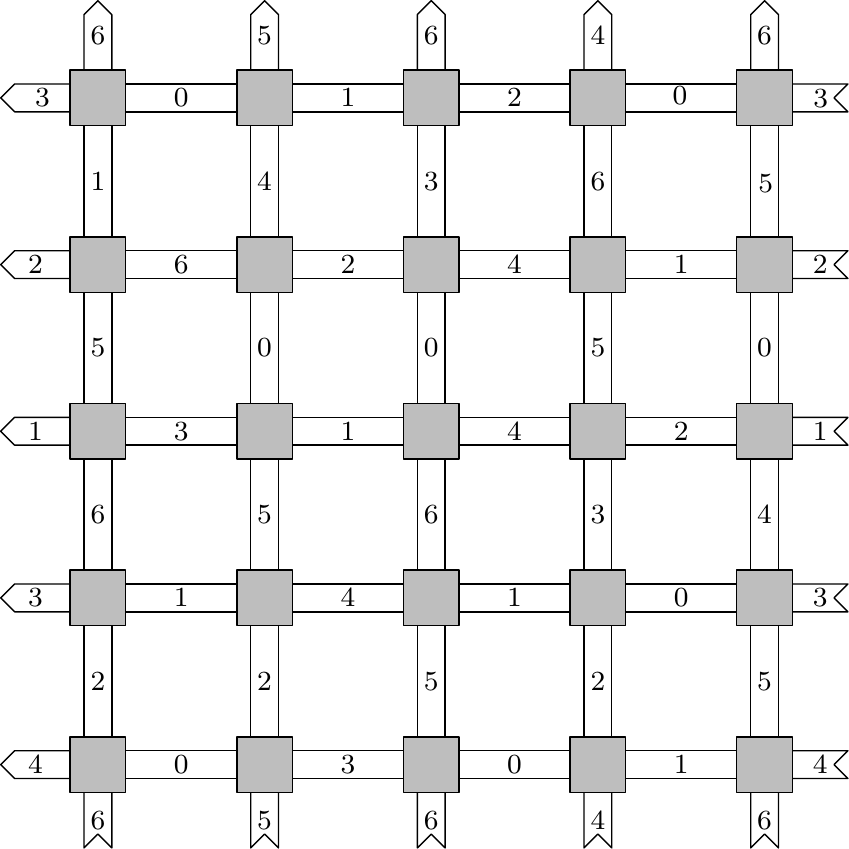}
		$$
		\caption{A star $7$-edge-coloring of $C_5 \cartp C_5$}
		\label{fig:C5C5}
	\end{subfigure}
	\begin{subfigure}[t]{.55\textwidth}
		$$
			\includegraphics[scale=\scaleConst]{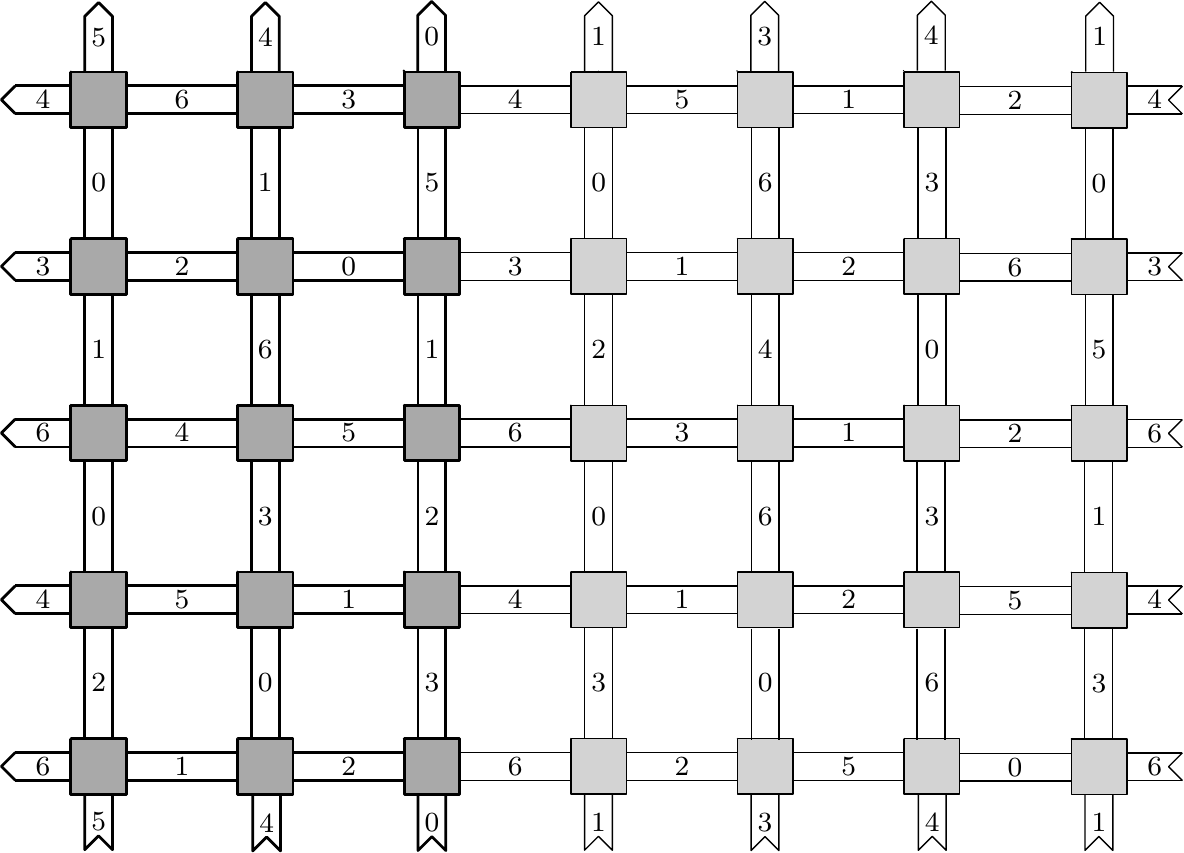}
		$$
		\caption{A star $7$-edge-coloring of $C_5 \cartp C_7$, 
			including a star edge-coloring of $C_5 \cartp C_3$ (darker vertices)}
		\label{fig:C5C7}
		\end{subfigure}
	
			\vspace{0,5cm}
		\begin{subfigure}[t]{\textwidth}
			$$
				\includegraphics[scale=\scaleConst]{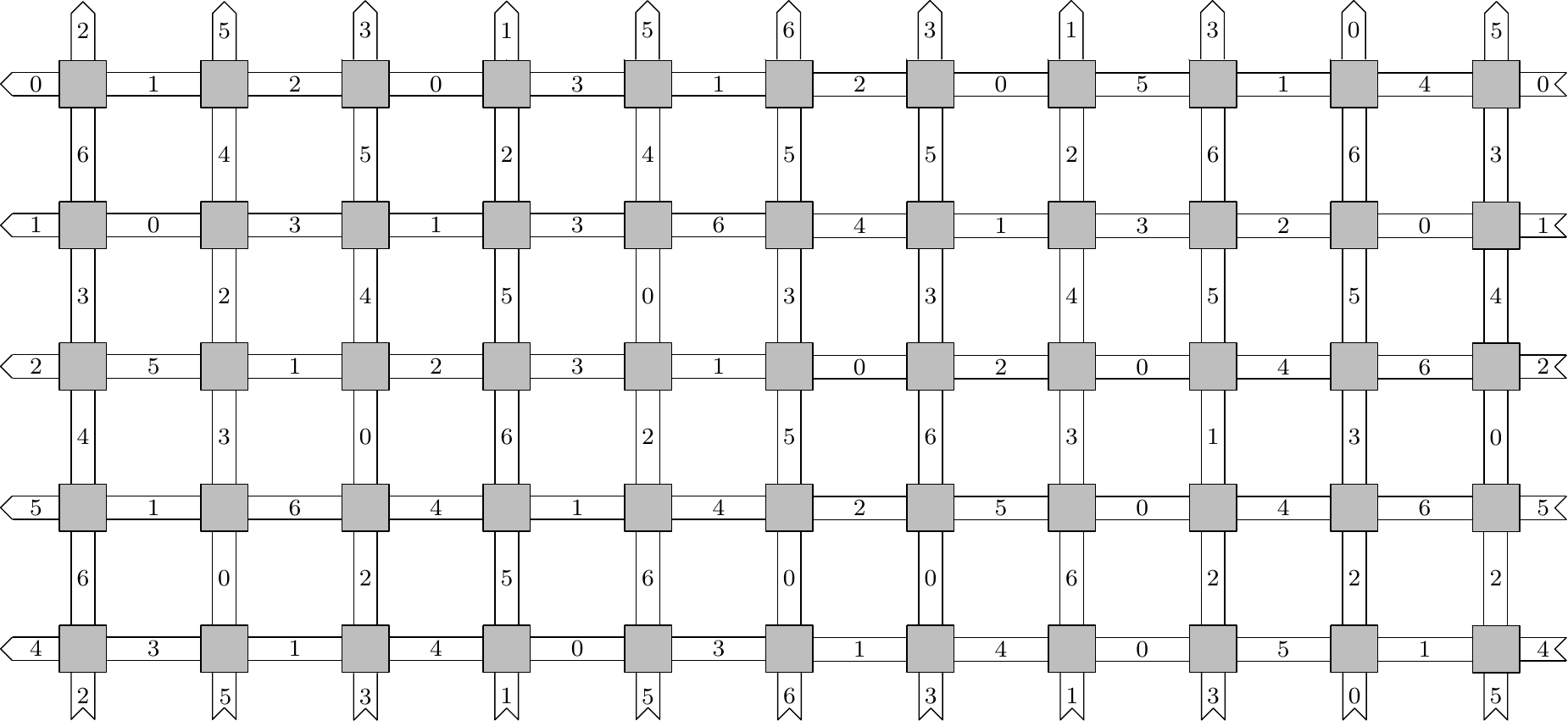}
			$$
		\caption{A star $7$-edge-coloring of $C_5 \cartp C_{11}$}
		\label{fig:C5C11}
		\end{subfigure}

\end{figure}

\begin{figure}[ht]
	\caption{A star $7$-edge-coloring of $C_7 \cartp C_7$ 
		including a star edge-coloring of $C_3 \cartp C_7$ (darker vertices in horizontal direction)
		and a star edge-coloring of $C_7 \cartp C_3$ (darker vertices in vertical direction)}	
	$$
		\includegraphics[scale=\scaleConst]{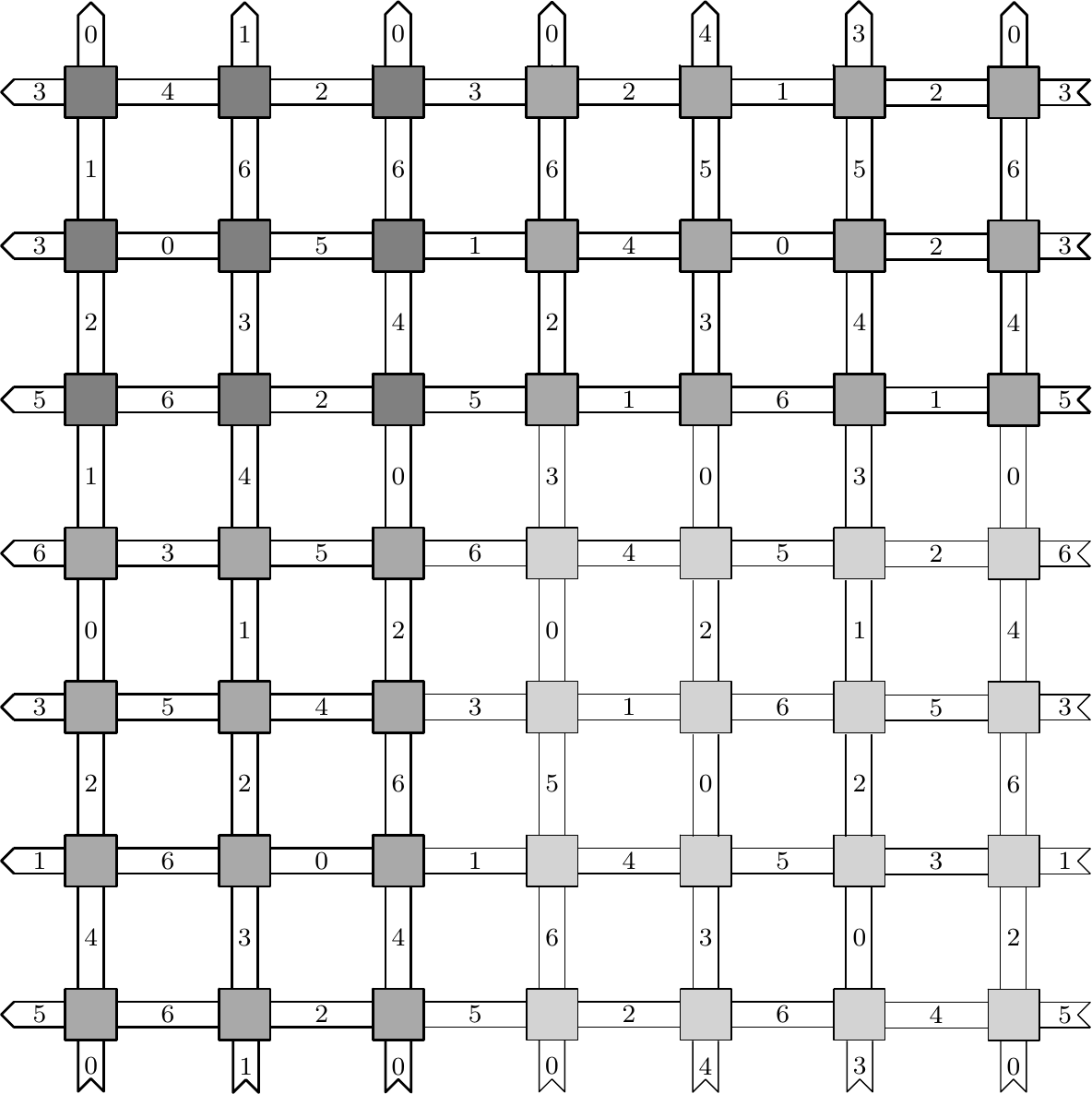}
	$$
	\label{fig:C7C7}
\end{figure}

\begin{figure}[ht]
	\caption{A star $7$-edge-coloring of $C_{11} \cartp C_{11}$
		including a star edge-coloring of $C_7 \cartp C_{11}$ (all darker vertices),
		which furthermore includes a star edge-coloring of $C_3 \cartp C_{11}$ (the darkest vertices above)}
	$$
		\includegraphics[scale=\scaleConst]{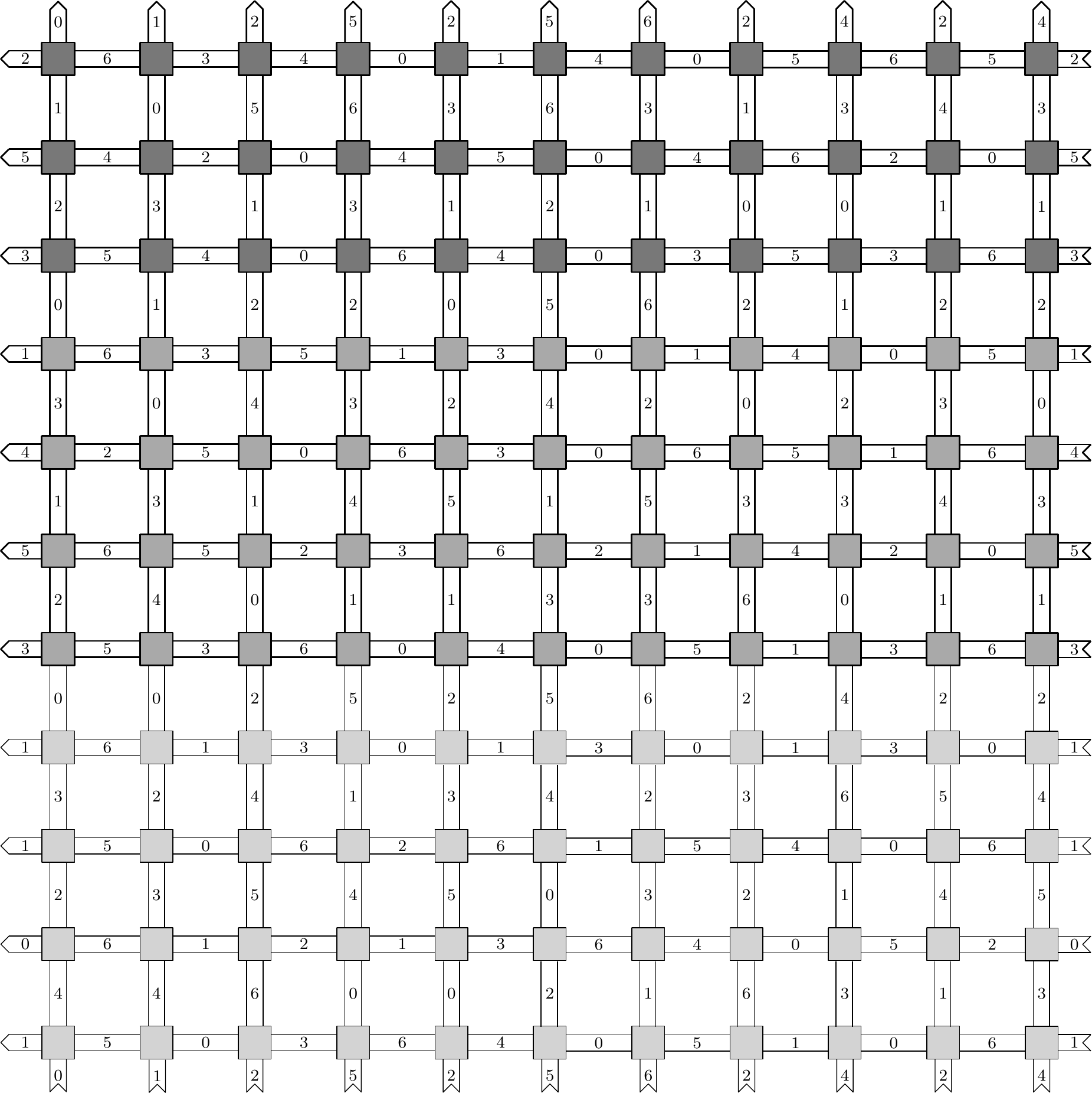}
	$$
	\label{fig:C11C11}
\end{figure}

\clearpage

\section{Cartesian products of paths and cycles}

\begin{figure}[ht!]
\caption{Cartesian products of cycles and $P_6$}
		\begin{subfigure}[t]{.5\textwidth}
			$$
				\includegraphics[scale=\scaleConst]{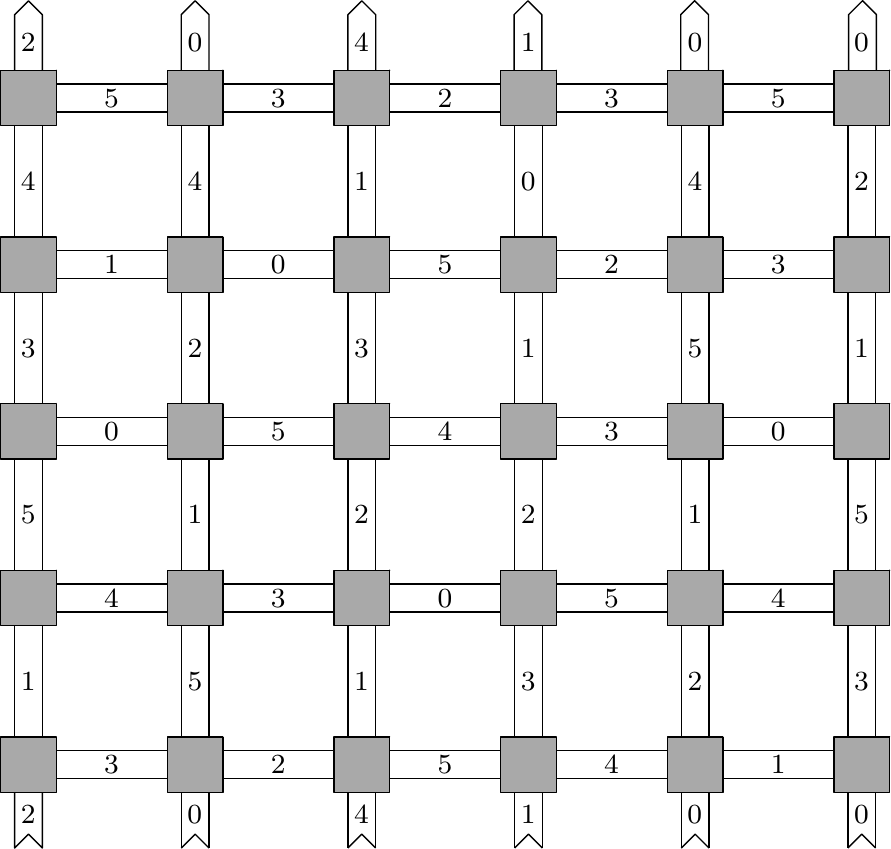}
			$$
			\caption{A star $6$-edge-coloring of $C_5 \cartp P_6$}
			\label{fig:C5P6}
		\end{subfigure}
		
		\begin{subfigure}[b]{.5\textwidth}
			$$				
				\includegraphics[scale=\scaleConst]{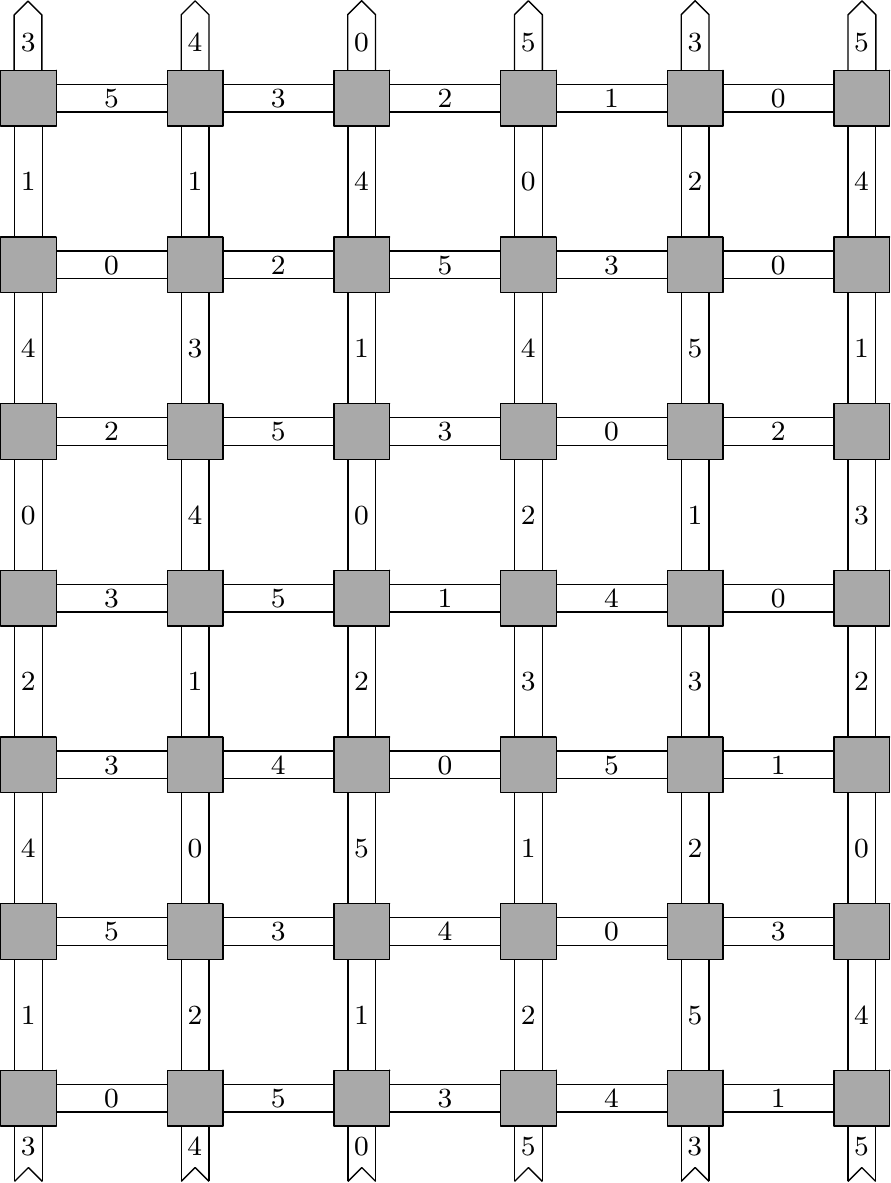}
			$$		
			\caption{A star $6$-edge-coloring of $C_7 \cartp P_6$}
			\label{fig:C7P6}
		\end{subfigure}
		\begin{subfigure}[b]{.5\textwidth}
			$$
				\includegraphics[scale=\scaleConst]{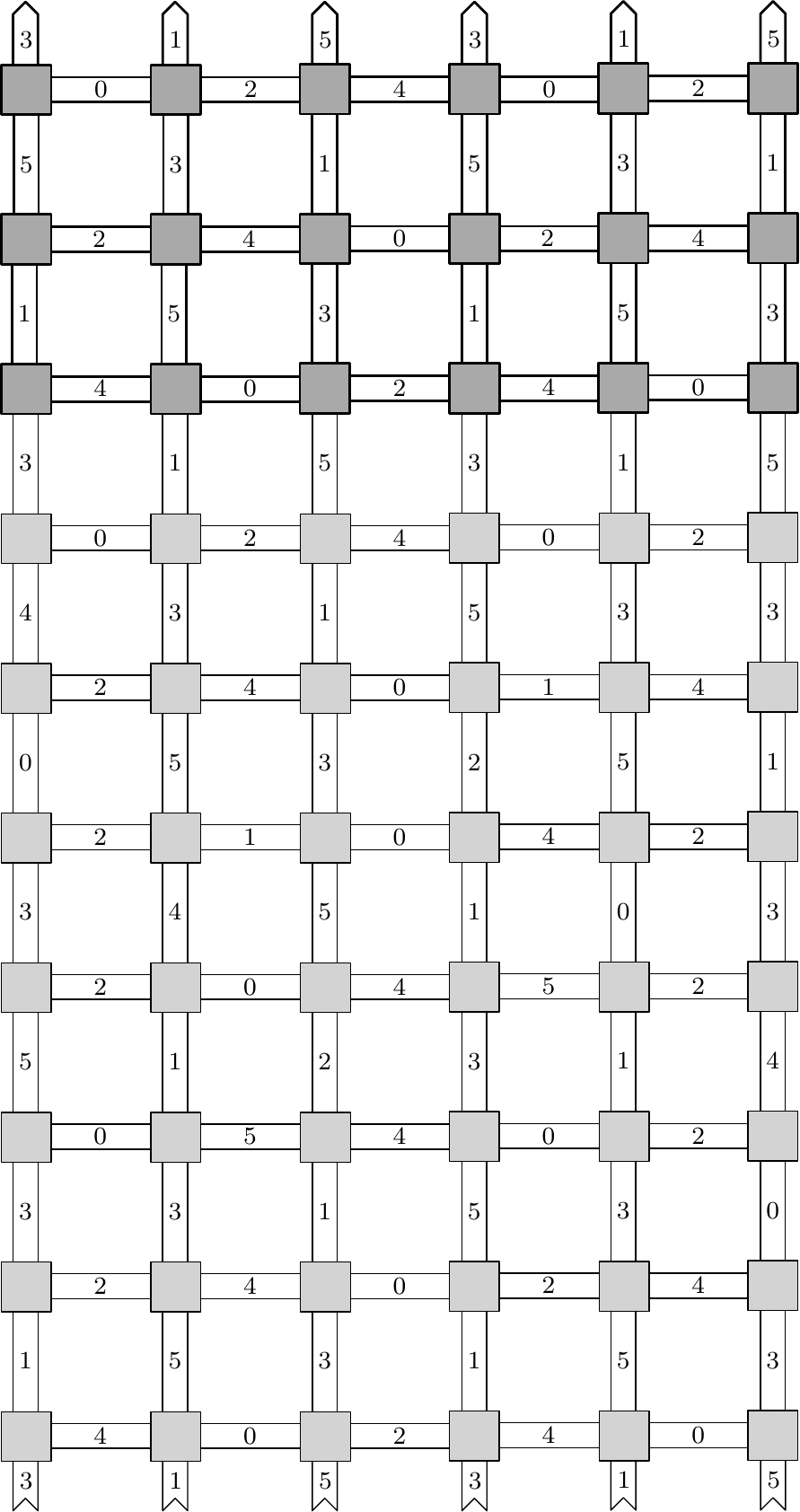}
			$$
			\caption{A star $6$-edge-coloring of $C_{10} \cartp P_6$ 
			including a star edge-coloring of $C_3 \cartp P_6$ (darker vertices)}
			\label{fig:P6C10}
		\end{subfigure}
\end{figure}

\begin{figure}[htp!]
\caption{Cartesian products of cycles and $P_8$}

		\begin{subfigure}[t]{.5\textwidth}
		$$
			\includegraphics[scale=\scaleConst]{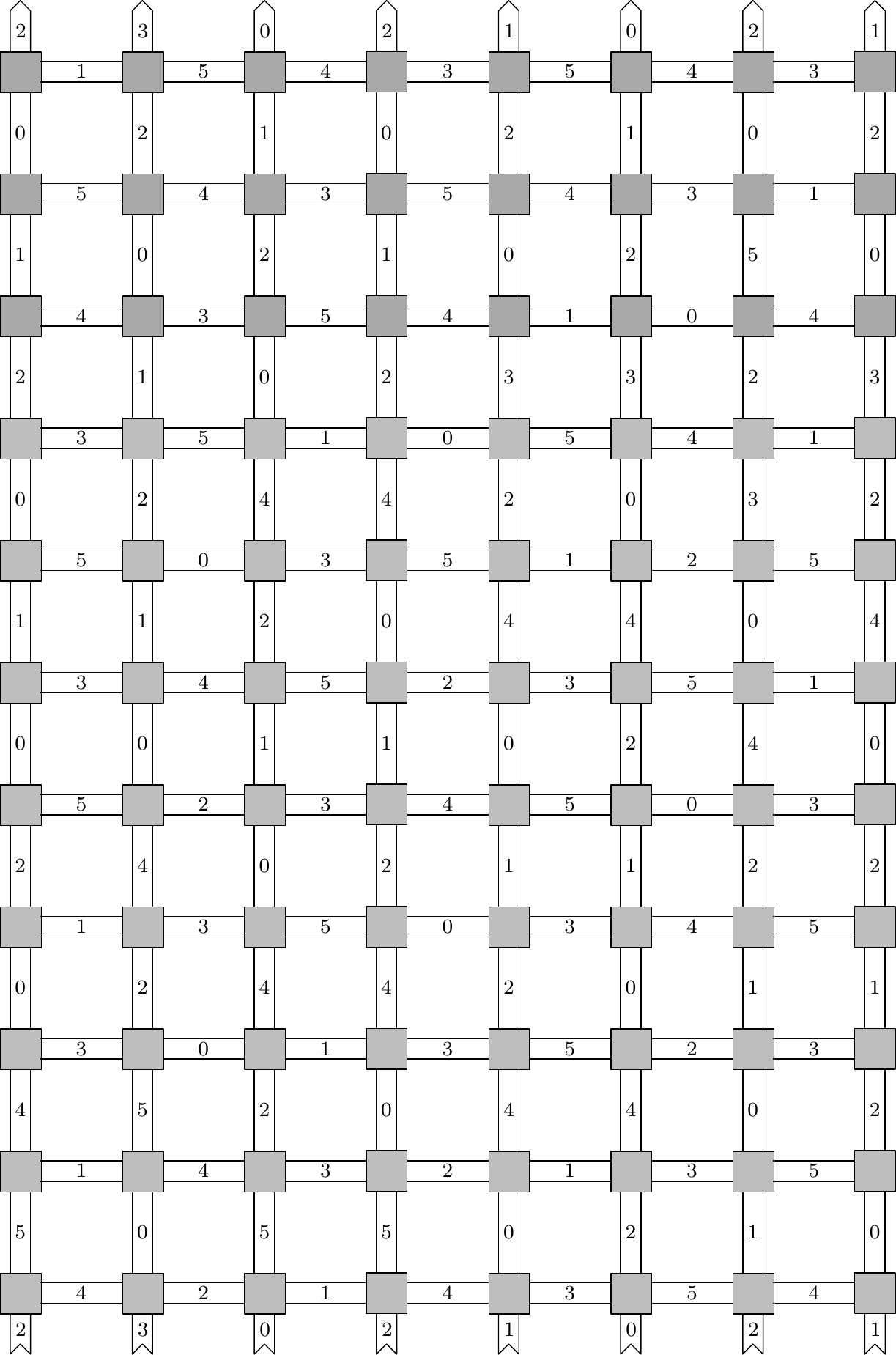}
		$$
		\caption{A star $6$-edge-coloring of $C_{11} \cartp P_8$}
		\label{fig:P8C11}	
		\end{subfigure}
		\begin{subfigure}[t]{.56\textwidth}
		$$
			\includegraphics[scale=\scaleConst]{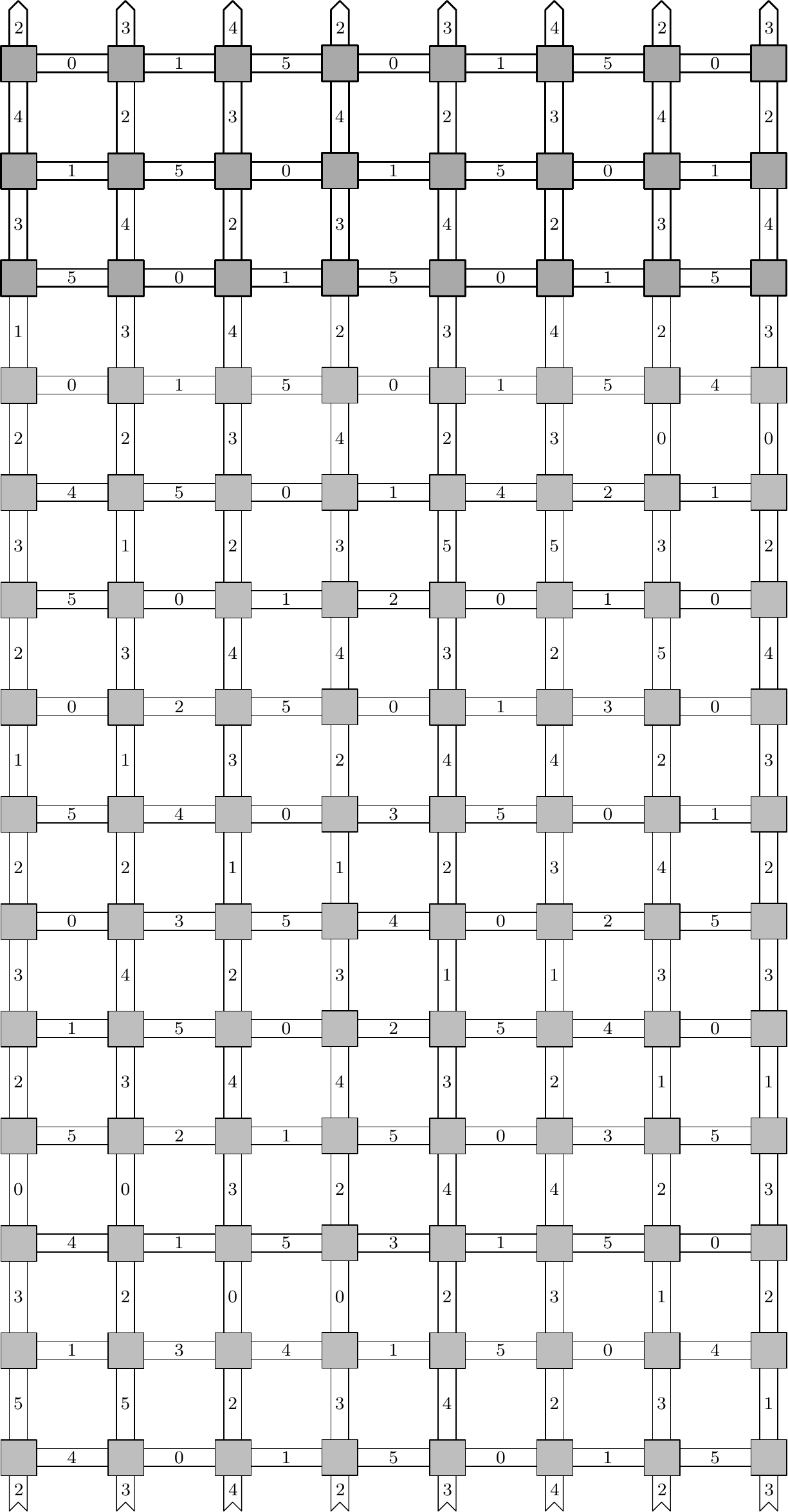}
		$$
		\caption{A star $6$-edge-coloring of $C_{14} \cartp P_8$ 
			including a star edge-coloring of $C_3 \cartp P_8$ (darker vertices)}
		\label{fig:P8C14}
		\end{subfigure}

\end{figure}


\end{document}